\documentclass[a4paper,11pt,reqno]{amsart}
\usepackage[centertags]{amsmath}
\usepackage{amsfonts,amssymb,amsthm} 
\usepackage[dvips]{graphicx} 
\usepackage{psfrag}
\usepackage[english]{babel}
\usepackage{mathrsfs}
\usepackage{esint}
\usepackage{color}
\usepackage[body={16cm,23cm},centering]{geometry} 
\usepackage[colorlinks,citecolor=blue,urlcolor=blue]{hyperref}
\usepackage{fancyhdr}
\newtheorem{theorem}{Theorem}
\newtheorem*{mtheorem}{Main Theorem}

\newtheorem{lemma}{Lemma}
\newtheorem{remark}{Remark}

\newtheorem{corollary}{Corollary}

\newcommand{\e}{\varrho}

\newcommand{\R}{\mathbb{R}}

\newcommand{\Ha}{\mathcal{H}}
\newcommand{\Hh}{\mathscr{H}} 

\newcommand{\Sf}{\mathbb{S}}
\newcommand{\F}{\mathcal{F}}

\renewcommand{\div}{\mathrm{div}}

\newcommand{\la}{\langle}
\newcommand{\ra}{\rangle}
\newcommand{\unit}{\omega}
\newcommand{\pa}{\partial}
\mathchardef\emptyset="001F

\title[Symmetry of minimizers]
{Symmetry of minimizers of a Gaussian \\ isoperimetric problem}
\author[M. Barchiesi and V. Julin]
{Marco Barchiesi and Vesa Julin}
\address[M. Barchiesi]{Universit\`{a} di Napoli Federico II,
Dipartimento di Matematica e Applicazioni,
Via Cintia, Monte Sant'Angelo, I-80126 Napoli, Italy}
\email{barchies@gmail.com}
\address[V. Julin]{University of Jyv\"{a}skyl\"{a},
Department of Mathematics and Statistics,
P.O.Box 35 (MaD) FI-40014, Finland}
\email{vesa.julin@jyu.fi}
\date{May 8, 2018}

\begin{document}
\maketitle

%%%%%%%%%%%%%%%%%%%%%%%%%%%%%%%%%%%%%%%%%%%%%%%%%%%%%%%%%%%%%%%%%%%%%%%%%%%%%%%%%%%%%%%%%%%%%%
\begin{center}
\begin{minipage}{13cm}
\small{
\noindent {\bf Abstract.} 
We study an isoperimetric problem described by a functional that consists of the standard Gaussian perimeter and  
the norm of the barycenter. This second term has a repulsive effect, and it is in competition with the perimeter.
Because of that, in general the solution is not the  half-space.
We characterize all the minimizers of this functional, when the volume is close to one, by proving that the minimizer 
is either the half-space or the symmetric strip, depending  on  the strength of the repulsive term.  
As a corollary, we obtain  that the symmetric strip is the  solution of the Gaussian isoperimetric problem
among symmetric sets when the volume is close to one.

}
\end{minipage}
\end{center}

\bigskip

%%%%%%%%%%%%%%%%%%%%%%%%%%%%%%%%%%%%%%%%%%%%%%%%%%%%%%%%%%%%%%%%%%%%%%%%%%%%%%%%%%%%%%%%%%%%%%%%%%%%%
\tableofcontents

%%%%%%%%%%%%%%%%%%%%%%%%%%%%%%%%%%%%%%%%%%%%%%%%%%%%%%%%%%%%%%%%%%%%%%%%%%%%%%%%%%%%%%%%%%%%%%%%%%%%%%%%%%%%%%%%%%%%%%%

\section{Introduction}

\noindent 
The Gaussian  isoperimetric inequality  (proved by Borell~\cite{Bor} and  Sudakov-Tsirelson \cite{SuCi})  
states that among all sets with given Gaussian measure the half-space has the smallest Gaussian perimeter. 
Since the half-space is not symmetric with respect to the origin, a natural question is to restrict the problem among sets 
which are symmetric, i.e., either central symmetric ($E = -E$) or coordinate wise symmetric ($n$-symmetric).
This problem  turns out to be rather difficult as every known method that  has been used to prove  the Gaussian   
isoperimetric inequality, such as symmetrization \cite{Ehr}  and the Ornstein-Uhlenbeck semigroup argument \cite{BL},   
seems to fail. In fact, at the moment  it is not even clear what   the solution to this problem should be.

The Gaussian  isoperimetric problem for symmetric sets or its generalization to Gaussian noise is stated as an open problem  
in \cite{CR, OP}. In the latter it was conjectured that the solution should be the ball or its complement, but this  was 
recently disproved in \cite{He}. Another natural candidate for the solution is the symmetric strip or its complement. 
Indeed, in \cite{Barthe01} Barthe proved that if one replaces the standard Gaussian perimeter by  a certain anisotropic perimeter,  
the solution of the isoperimetric problem among $n$-symmetric sets is the symmetric strip or its complement. 
We mention also a somewhat similar result by Latala and Oleszkiewicz  \cite[Theorem 3]{LO} who proved that the symmetric 
strip minimizes the Gaussian perimeter weighted with the width of the set among 
convex and symmetric sets with volume constraint.  For the standard perimeter the problem is more difficult 
as a simple energy comparison shows (see \cite{He2}) that when the volume is exactly one half, the two-dimensional disk  
and the three-dimensional ball have both  smaller perimeter than the symmetric strip in dimension two and three, respectively. 
Similar difficulty appears also in the isoperimetric problem on sphere for symmetric sets, where it is known that the union of two 
spherical caps does not always have the smallest surface area (see \cite{Barthe01}). However, it might still be the case that the 
 solution of the problem is a cylinder $B_r^k \times \R^{n-k}$, or its complement, for some $k$ depending on the volume
(see \cite[Conjecture 1.3]{He2}). Here $B_r^k$ denotes the $k$-dimensional ball with radius $r$. At least the results 
by Heilman \cite{He,He2} and La Manna \cite{LaM} seem to indicate this. 

To the best of the authors knowledge there are no other results directly related to this problem. 
In \cite{CM}  Colding and Minicozzi introduced the Gaussian entropy, which is defined for sets as
\[
\Lambda(\pa E) = \sup_{x_0 \in \R^n, t_0 >0}P_\gamma(t_0^{-1}(E-\{x_0\})),
\]
where $P_\gamma$ is the Gaussian perimeter defined below. The Gaussian entropy is important as it is decreasing under the 
mean curvature flow and for this reason  in \cite{CM}  the authors studied sets which are stable for the Gaussian entropy.   
It was conjectured in \cite{CIMW} that the sphere minimizes the entropy among closed hypersurfaces (at least  in low dimensions). 
This was proved by Bernstein and Wang in \cite{BW}  in low dimensions and more recently by Zhu \cite{Zhu}  in every dimension. 
This problem is related to the symmetric Gaussian problem since the Gaussian entropy of a self-shrinker equals its Gaussian perimeter.

In this paper we prove that the symmetric strip is the solution of the Gaussian isoperimetric problem for symmetric set when 
the volume is close to one. (Similarly, its complement is the solution when the volume is close to zero). Our proof is direct 
and thus we could  give an explicit estimate on how close to one the volume has to be. In particular, the bound on the volume is independent of the 
dimension. But as our proof is rather long and the  
bound on the volume is obtained after numerous inequalities, we prefer to state the result in a more qualitative way in order 
to avoid heavy computations.

In order to describe the main result  more precisely, we introduce our setting.  Given a Borel set $E\subset\R^n$, $\gamma(E)$ 
denotes its Gaussian measure, defined as
\begin{equation*}
\gamma(E):=\frac{1}{(2\pi)^\frac{n}{2}}\int_E e^{-\frac{|x|^2}{2}}dx.
\end{equation*}
If $E$ is an open set with Lipschitz boundary, $P_\gamma(E)$ denotes its \emph{Gaussian perimeter}, defined as 
\begin{equation}\label{gauss perim}
P_\gamma(E):=\frac{1}{(2\pi)^\frac{n-1}{2}}\int_{\partial E}e^{-\frac{|x|^2}{2}}d\Ha^{n-1}(x),
\end{equation}
where $\Ha^{n-1}$ is the $(n-1)$-dimensional Hausdorff measure. We define the (non-renormalized) barycenter  of a set $E$ as
\[
b(E) := \int_E x \, d\gamma(x)
\]
and  define the function $\phi:\R\rightarrow(0,1)$ as 
\begin{equation*}
\phi(s):=\frac{1}{\sqrt{2\pi}}\int_{-\infty}^s e^{-\frac{t^2}{2}}dt.
\end{equation*}
Moreover, given $\unit \in\Sf^{n-1}$ and $s \in \R$,  $H_{\unit,s}$ denotes the half-space of the form
\begin{equation*}
H_{\unit,s}:=\{x\in\R^n \text{ : } \la x,\unit\ra < s\},
\end{equation*}
while $D_{\unit,s}$ denotes the symmetric strip 
\begin{equation*}
D_{\unit,s}:=\{x\in\R^n \text{ : }  | \la x,\unit\ra| <a(s)\}, 
\end{equation*}
where   $a(s)>0$ is chosen such that $\gamma(H_{\unit,s})=\gamma(D_{\unit,s})$.

We approach the problem by studying  the minimizers   of the functional
\begin{equation}\label{key problem}
\F(E):=P_\gamma(E)+  \e \, \sqrt{\pi/2}\,  |b(E)|^2 
\end{equation} 
under the volume constraint  $\gamma(E) = \phi(s)$. Note that the isoperimetric inequality implies that for $\e=0$ the half-space 
is the only minimizer of \eqref{key problem}, while  it is easy to see that the quantity $|b(E)|$ is maximized by the half-space. 
Therefore the two terms in \eqref{key problem} are in competition and we call the barycenter term repulsive, as it prefers to balance 
the volume around the origin. It is  proven in \cite{Bar14, EL} that when $\e$ is small, the half-space is still  the  only minimizer 
of \eqref{key problem}. This result implies the quantitative Gaussian isoperimetric inequality (see also \cite{CFMP, MN, MN2, Bar16}). 
It is clear that when we keep increasing the value $\e$, there is a threshold, say $\e_s$, such that for $\e >\e_s$ the half-space  
$H_{\unit,s}$ is no longer the minimizer of \eqref{key problem}. In this paper we are interested in characterizing the minimizers of    
\eqref{key problem} after this threshold. Our main result reads as follows. 

\begin{mtheorem}
There exists $s_0>0$ such that the following holds: when $s \geq s_0$ there is a threshold $\e_s$ such that for $\e\in[0,\e_s)$ the minimizer
of \eqref{key problem} under  volume constraint  $\gamma(E) = \phi(s)$ is the half-space $H_{\omega,s}$, while for $\e\in(\e_s, \infty)$ 
the minimizer is the symmetric strip $D_{\omega,s}$.
\end{mtheorem}

As a corollary this  provides the solution for the symmetric Gaussian problem, because symmetric sets have barycenter zero. 
\begin{corollary}
There exists $s_0>0$ such that for  $s \geq s_0$ it holds 
\begin{equation*}
P_\gamma(E)\geq P_\gamma(D_{\omega,s}) = \Bigl( 1 + \frac{\ln 2}{s^2} +  o(1/s^2) \Bigr) e^{-\frac{s^2}{2}},
\end{equation*}
for any symmetric set $E$ with volume $\gamma(E)=\phi(s)$, and the equality holds if and only if $E=D_{\omega,s}$ for some $\unit\in\Sf^{n-1}$.
\end{corollary}

\medskip
Another corollary of the theorem is  the optimal constant in  the quantitative Gaussian 
isoperimetric inequality (see  \cite{Bar14, EL}) when the volume is close to one. Let us denote by $\beta(E)$ the strong asymmetry
\begin{equation*}
\beta(E):=\min_{\unit\in\Sf^{n-1}}\big|b(E)-b(H_{\unit,s})\big|,
\end{equation*}
which measures the distance between a set $E$  and the family of half-spaces.

\begin{corollary} \label{quanti}
There exists $s_0>0$ such that for  $s \geq s_0$ it holds
\begin{equation*}
P_\gamma(E) - P_\gamma(H_{\unit,s}) \geq  c_s \beta(E), 
\end{equation*}
for every set  $E$ with volume $\gamma(E)=\phi(s)$. The optimal constant is given by 
\begin{equation*}
c_s = \sqrt{2 \pi } \,e^{s^2/2} \, \left( P_\gamma(D_{\omega,s}) - P_\gamma(H_{\unit,s}) \right) 
=  \sqrt{2 \pi } \,  \frac{\ln 2}{s^2} +  o(1/s^2). 
\end{equation*}
\end{corollary}

\medskip
It would be interesting to obtain a result analogous to Corollary \ref{quanti}  in the Euclidean setting, where the minimization 
problem which corresponds   to \eqref{key problem} was introduced in \cite{FJ}, and on the sphere \cite{BDF}. 
The motivation for this is that, by the result of the second author \cite{vesku}, the optimal constant for the quantitative 
Euclidean isoperimetric inequality implies an estimate on the range of volume where the ball is the minimizer of the Gamov's 
liquid drop model \cite{Gamov}. This is a classical model used in nuclear physics and has gathered a lot  attention in  
mathematics  in recent years \cite{CMT, CP, KM}.  We also refer to the survey paper \cite{F}  for the state-of-the-art in the 
quantitative  isoperimetric  and other functional inequalities.    

The main idea of the proof is to study the functional \eqref{key problem} when the parameter $\e$ is within a carefully chosen range 
 $(\e_{s,1},\e_{s,2})$, and to prove that within this range the only local minimizers, which satisfy certain perimeter bounds,  are 
the half-space $H_{\unit,s}$ and the symmetric strip $D_{\unit,s}$. We have to  choose the lower bound $\e_{s,1}$  large enough so 
that the symmetric strip is a local minimum of \eqref{key problem}. On the other hand, we have to choose  the upper bound  $\e_{s,2}$ 
small enough  so that no other local minimum than $H_{\unit,s}$ and  $D_{\unit,s}$ exist. Naturally also the threshold value $\e_s$ has 
to be within the range $(\e_{s,1},\e_{s,2})$. 

Our proof is based on reduction argument where we reduce the dimension of the problem from $\R^n$ to $\R$. First, we develop further 
our ideas from \cite{Bar14} to   reduce the problem from $\R^n$ to $\R^2$ by  a rather short argument.  In this step it is crucial that  
we are not constrained to keep the sets symmetric.  The main challenge is  thus to prove the theorem in $\R^2$, since here  
we cannot apply the previous reduction argument anymore. Instead, we use an ad-hoc argument to reduce the problem from $\R^2$ to $\R$ 
essentially by PDE type estimates from the Euler equation and from the stability condition. We give an independent overview of 
this argument at the beginning of the proof of Theorem \ref{2 to 1} in Section \ref{section redux 2 to 1}. 
Finally, we solve the problem in $\R$ by a direct argument.

%%%%%%%%%%%%%%%%%%%%%%%%%%%%%%%%%%%%%%%%%%%%%%%%%%%%%%%%%%%%%%%%%%%%%%%%%%%%%%%%%%%%%%%%%%%%%%%%%%%%%%%%%%%%%%%%%%%%

\section{Notation and set-up}\label{notation}

\noindent
In this section we briefly introduce our notation and discuss about preliminary results and  estimates. 
\emph{We remark  that throughout the paper the parameter $s$, associated with the volume, is assumed to be large 
even if not explicitly mentioned.  In particular, our estimates are understood to hold when  $s$ is chosen to be large enough.}
$C$ denotes a numerical constant which may vary from line to line.

We denote the $(n-1)$-dimensional Hausdorff measure with Gaussian weight by $\Ha^{n-1}_\gamma$, i.e.,  for every Borel set $A$ we define
\begin{equation*}
\Ha^{n-1}_\gamma(A):=\frac{1}{(2\pi)^\frac{n-1}{2}}\int_{A}e^{-\frac{|x|^2}{2}}d\Ha^{n-1}.
\end{equation*}
We minimize the functional  \eqref{key problem} among sets with locally finite perimeter and have the existence of a minimizer for 
every $\e$ by  an argument similar to \cite[Proposition 1]{Bar14}. If $E \subset \R^n$ is a set  of locally finite perimeter we denote 
its reduced boundary by $\partial^*E$ and define its Gaussian perimeter by 
\[
P_\gamma(E) := \Ha^{n-1}_\gamma(\partial^* E).
\]
We denote the generalized exterior normal by $\nu^E$ which is defined on  $\partial^* E$. As introduction to the theory of sets of 
finite perimeter and perimeter minimizers we refer to \cite{Ma}. 

If the reduced boundary $\partial^* E$  is a smooth hypersurface we denote  the second fundamental form by $B_E$ and the mean curvature 
by $\Hh_E$, which for us is the sum of the principle curvatures. We adopt the notation from \cite{Giusti} and define the tangential gradient 
of a function $f$, defined in a neighborhood of  $\partial^* E$, by $\nabla_\tau f := \nabla f - (\nabla f \cdot \nu^E) \nu^E$. 
Similarly, we define the tangential divergence of a vector field by $\div_\tau X := \div X - \la DX \nu^E, \nu^E \ra  $  and the 
Laplace-Beltrami operator as $\Delta_\tau f := \div_\tau (\nabla_\tau f)$. 
The divergence theorem on $\partial^* E$   implies that for every vector field $X \in C_0^1(\partial^* E ; \R^n)$ it holds 
\[
\int_{\partial^* E } \div_\tau X \, d \Ha^{n-1} = \int_{\partial^* E }\Hh_E \la X , \nu^E\ra  \, d \Ha^{n-1}. 
\]
If $\partial^* E$  is a smooth hypersurface, we may extend any function $f \in C_0^1(\partial^* E)$ to a neighborhood of  
$\partial^* E$ by the distance function. For simplicity we will omit to indicate the dependence on the set $E$ when this is clear, 
by simply writing $\nu=\nu^E$, $\Hh=\Hh_E$ etc... 

We denote the mean value of a function $f : \partial^* E \to \R$ by 
\[
\bar{f} := \fint_{\partial^* E} f \, \Ha^{n-1}_\gamma,
\]
and  its  average over a subset $\Sigma \subset \partial^* E$ by 
\[
(f)_{\Sigma} := \fint_{\Sigma} f \, \Ha^{n-1}_\gamma. 
\]
We recall that  for every number $a \in \R$ it holds
\[
\int_{\Sigma} (f - (f)_{\Sigma})^2 \, d \Ha_\gamma^1 \leq \int_{\Sigma} (f - a)^2 \, d \Ha_\gamma^1.
\]

Recall that $H_{\unit,s}$ denotes the half-space  $\{x\in\R^n \text{ : } \la x,\unit\ra < s\}$ and $D_{\unit,s}$ denotes the symmetric strip 
$\{x\in\R^n \text{ : }  | \la x,\unit\ra| <a(s)\}$, where $a(s)$ is chosen such that   $\gamma(D_{\unit,s})  = \gamma(H_{\unit,s}) =  \phi(s)$.   
Since we are assuming that $s$ is large, it is important to know the asymptotic behavior of  the quantities
$\phi(s)$, $a(s)$, and $P(D_{\unit,s})$. A simple analysis shows that 
\begin{equation}\label{estimate1}
\phi(s) =  1- \frac{1}{\sqrt{2 \pi}} \left(\frac{1}{s}+ o(1/s^2) \right) e^{-\frac{s^2}{2}} .
\end{equation}
The asymptotic behavior of $a(s)$ is a slightly more complicated. We will  show that
\begin{equation}\label{estimate2}
a(s) = s + \frac{\ln 2}{s} + o(1/s).
\end{equation}
To this aim  we write $a(s)= s + \delta(s)$, so that \eqref{estimate2} is equivalent to $\lim_{s\to\infty}s\delta(s)=\ln 2$.
 We argue by contradiction  and assume that $\lim_{s\to\infty} s\delta(s)>\ln 2$. By the volume constraint
\begin{equation}\label{balance mass}
2\int_{a(s)}^\infty e^{-\frac{t^2}{2}}\, dt =\int_{s}^\infty e^{-\frac{t^2}{2}}\, dt
\end{equation}
and therefore
\[
\begin{split}
1< \lim_{s \to \infty} 
\frac{2\int_{s+\ln 2/s}^\infty e^{-\frac{t^2}{2}}\, dt}{\int_{s}^\infty e^{-\frac{t^2}{2}}\, dt}
=\lim_{s \to \infty} \frac{-2(1-\frac{\ln 2}{s^2}) e^{-\frac{(s + \ln 2/s)^2}{2}}}{-e^{-\frac{s^2}{2}}} = 1,
\end{split}
\]
which is a contradiction. We arrive to a similar  contradiction if $\lim_{s\to\infty} s\delta(s)<\ln 2$. Therefore we have \eqref{estimate2}. 

For the perimeter of the strip $D_{\unit,s}$ we have the following estimate:
\begin{equation}\label{estimate PD}
P_\gamma(D_{\unit,s}) = \Bigl( 1 + \frac{\ln 2}{s^2} +  o(1/s^2) \Bigr) e^{-\frac{s^2}{2}} .
\end{equation}
Indeed, since $P_\gamma(D_{\unit,s})=2e^{-a(s)^2/2}$, by differentiating \eqref{balance mass} 
we get $(1+\delta')P_\gamma(D_{\unit,s})=e^{-s^2/2}$. Moreover, by $\lim_{s\to\infty}s\delta(s)=\ln 2$
we get $\lim_{s\to\infty}s^2\delta'(s)=-\ln 2$ and thus
\begin{equation*}\begin{split}
&\lim_{s\to\infty}s^2e^{\frac{s^2}{2}} \Bigl(P_\gamma(D_{\unit,s})-\Bigl( 1 + \frac{\ln 2}{s^2}\Bigr) e^{-\frac{s^2}{2}}\Bigr)\\
&=-\lim_{s\to\infty}s^2\Bigl(\frac{\delta'}{1+\delta'}+\frac{\ln 2}{s^2}\Bigr)=0.
\end{split}\end{equation*}
In particular, according to our main theorem   the threshold value  has the asymptotic behavior 
\begin{equation}\label{choice epsilon down}
\e_s = 2 \ln 2 \,  \frac{ \sqrt{2 \pi}}{ s^2}  e^{\frac{s^2}{2}} \left( 1+ o(1) \right).
\end{equation}
This follows from the fact that threshold value $\e_s$ is the unique value of $\e$ for which the functional \eqref{key problem} 
satisfies   $\mathcal{F} (H_{\unit,s})=  \mathcal{F} (D_{\unit,s})$, i.e., 
\[
P_\gamma(D_{\unit,s}) = e^{-\frac{s^2}{2}} + \frac{\e_s}{2 \sqrt{2 \pi}}  e^{-s^2}
\]
by taking into account that $|b(H_{\unit,s})|=e^{-s^2/2}/\sqrt{2\pi}$.

In order to simplify the upcoming technicalities we  replace the volume constraint in the original functional \eqref{key problem} 
with a volume penalization. We redefine  $\F$  for any set of locally finite perimeter  as
\begin{equation}\label{functional}
\mathcal{F} (E) :=  P_\gamma(E)+ \e\sqrt{\pi/2}\,|b(E)|^2 
+ \Lambda\sqrt{2\pi}\, \big| \gamma(E) - \phi(s) \big|,
\end{equation}
where we choose 
\begin{equation}\label{choice lambda}
\Lambda= s + 1.
\end{equation}
As with the original functional the existence of a minimizer of \eqref{functional} follows from \cite[Proposition 1]{Bar14}.  
It turns out that the minimizers of \eqref{functional} are the same as the minimizers of \eqref{key problem} under the  volume 
constraint $\gamma(E) = \phi(s)$, as proved in the last section. The advantage of a volume 
penalization is that it helps us to bound the Lagrange multiplier in a simple way. The  constants $\sqrt{\pi/2}$ and 
$\sqrt{2\pi}$ in front of the last two terms are chosen to simplify the formulas of the Euler equation and the second variation. 

As we explained in the introduction, the idea is to restrict  the parameter $\e$ in \eqref{functional} within a range, which 
contains the threshold  value \eqref{choice epsilon down} and such that the only local minimizers of \eqref{functional}, which 
satisfy certain perimeter bounds, are the half-space and the symmetric strip. To this aim we assume from now on that $\e$ is in the range
\begin{equation}\label{choice epsilon up}
\frac65 \, \frac{\sqrt{2 \pi}}{s^2}  e^{\frac{s^2}{2}} \leq  \e \leq  \frac75 \,  \frac{\sqrt{2 \pi}}{s^2}  e^{\frac{s^2}{2}}.
\end{equation}
Note that the threshold value \eqref{choice epsilon down} is within  this interval.   
 If we are able to show that when $\e$ satisfies \eqref{choice epsilon up} the only local minimizers of  \eqref{functional}  
are  $H_{\unit,s}$ and $D_{\unit,s}$, we obtain  the main result. Indeed,  when $\e$ takes the lower value in \eqref{choice epsilon up} 
it holds $\mathcal{F}(H_{\unit,s})< \mathcal{F}(D_{\unit,s})$ and   the minimizer is $H_{\unit,s}$. It is then not difficult to see that  
for every value $\e$ less than this, the minimizer is still $H_{\unit,s}$. Similarly,  when $\e$ takes the larger value in \eqref{choice epsilon up} 
it holds $\mathcal{F}(D_{\unit,s})< \mathcal{F}(H_{\unit,s})$ and  the minimizer is $D_{\unit,s}$. 
Hence, for every value $\e$ larger than this,  $D_{\unit,s}$ is still the minimizer of \eqref{functional}, since it has  barycenter zero. 

Next we deduce a priori perimeter  bounds for the minimizer. First, we may bound the perimeter from above by the minimality as
\begin{equation}\label{bound perimeter up}
P_\gamma(E)\leq\F(E)\leq\F(D_{\unit,s})=P_\gamma(D_{\unit,s})\leq \Bigl( 1 + \frac{1}{s^2}\Bigr) e^{-\frac{s^2}{2}}.
\end{equation} 
To bound the perimeter from below is slightly more difficult. Let $E$ be a  minimizer of \eqref{functional} 
with volume $\gamma(E)  = \phi(\bar{s})$. First, it is clear that $\bar{s} \geq 0$. Let us show that $\bar{s} \leq s + \tfrac{1}{s}$, 
which by the   Gaussian  isoperimetric inequality implies the following perimeter lower bound 
\begin{equation}\label{bound perimeter low}
P_\gamma(E) \geq \frac14 e^{-\frac{s^2}{2}}.
\end{equation} 
We argue by  contradiction and assume $\bar{s} > s + \tfrac{1}{s}$. By the Gaussian  isoperimetric inequality we deduce
\[
\mathcal{F} (E) \geq P_\gamma(E) + \Lambda\sqrt{2\pi}\, (\phi(s) - \phi(\bar{s})) 
\geq e^{-\frac{\bar{s}^2}{2}} + (s + 1) \int_s^{\bar{s}} e^{-\frac{t^2}{2}} \, dt. 
\]
Define the function $f:[s, \infty) \to \R$, $f(t) := e^{-\frac{t^2}{2}} + (s + 1) \int_s^{t} e^{-\frac{l^2}{2}} \, dl. $ 
By differentiating we get 
\[
f'(t) = (-t + s + 1) e^{-\frac{t^2}{2}}. 
\]
The function is clearly increasing up to $t = s + 1$ and then decreases to the value 
$\lim_{t \to \infty}f(t) = (s + 1) \int_s^{\infty} e^{-\frac{t^2}{2}} \, dt  \geq  \left( 1 + \frac{1}{2s} \right) e^{-\frac{s^2}{2}}$ 
by \eqref{estimate1}.  We also deduce that $f'(t) \geq \frac{1}{4}e^{-\frac{s^2}{2}}$ for 
$s \leq t \leq s + 1/s$ and therefore    $f(s + 1/s) \geq \left( 1 + \frac{1}{4s} \right) e^{-\frac{s^2}{2}}$. 
Hence,  if  $\bar{s} \geq s + \tfrac{1}{s}$ we have that 
\[
\mathcal{F} (E) \geq f(\bar{s}) \geq  \min \{f(s + 1/s), \lim_{t \to \infty}f(t)  \} \geq \left( 1 + \frac{1}{4s} \right) e^{-\frac{s^2}{2}}. 
\]
But this contradicts $\mathcal{F} (E) \leq \mathcal{F} (D_{\unit,s}) = P(D_{\unit,s})$ by  \eqref{estimate PD}. 
Thus we  have  \eqref{bound perimeter low}. 

\medskip
For reader's convenience we summarize the results  concerning the regularity of minimizers and the first and the second variation 
of \eqref{functional} contained in \cite[Section 4]{Bar14} in the following theorem.

\begin{theorem}\label{thm reg}
Let  $E$ be a minimizer of  \eqref{functional}. Then the reduced boundary  $\partial^*E$ is a relatively open, 
smooth hypersurface and satisfies the Euler equation
\begin{equation}\label{euler}
\Hh -\langle x, \nu \rangle + \e \langle b, x \rangle = \lambda \qquad \text{on }\, \partial^*E.
\end{equation}
The Lagrange multiplier $\lambda$ can be estimated by $|\lambda| \leq \Lambda$.
The singular part of the boundary $\partial E \setminus \partial^*E$ is empty when $n <8$, 
while for $n \geq 8$ its Hausdorff dimension can be estimated by $\dim_{\Ha}(\partial E \setminus \partial^*E) \leq n-8$.  
Moreover, the quadratic form associated with the second variation is non-negative
\begin{equation}\label{second var}\begin{split}
\F[\varphi] :=& \int_{\partial^* E} \left(|\nabla_\tau \varphi|^2 - |B_E|^2\varphi^2 
+ \e \langle b , \nu \rangle \varphi^2- \varphi^2\right)\, d\Ha^{n-1}_\gamma \\
& + \frac{\e}{\sqrt{2\pi}} \, \Big| \int_{\partial^* E} \varphi\, x \, d\Ha^{n-1}_\gamma \Big|^2 \geq 0
\end{split}\end{equation}
for every $\varphi \in C_0^\infty(\partial^* E)$ which satisfies 
$\int_{\partial^* E} \varphi \, d\Ha^{n-1}_\gamma = 0$.
\end{theorem}

The Euler equation \eqref{euler} yields important geometric equations for the position vector $x$ and for the Gauss map $\nu$.  
For arbitrary  $\omega \in  \mathbb{S}^{n-1}$ we write 
\[
x_\omega = \la x, \omega\ra \qquad \text{and} \qquad  \nu_\omega = \la \nu, \omega\ra.
\] 
If $\{e^{(1)},\ldots,e^{(n)}\}$ is a canonical basis of $\R^n$ we simply write    
\[
x_i = \la x, e_i \ra \qquad \text{and} \qquad  \nu_i = \la \nu, e_i \ra.
\]
From  \eqref{euler} and from the fact $\Delta_\tau x_\omega= - \Hh  \nu_\omega$ \cite[Proposition 1]{L} we have 
\begin{equation}\label{eq x 2}
\Delta_\tau x_\omega - \la \nabla_\tau x_\omega, x\ra =-x_\omega -\lambda \nu_\omega  + \e \la b,x\ra \nu_\omega.
\end{equation}
Moreover, from \eqref{euler} and from the fact  $\Delta_\tau \nu_\omega= - |B_E|^2 \nu_\omega + \la \nabla_\tau \Hh, \omega\ra$ 
 \cite[Lemma  10.7]{Giusti} we get
\begin{equation}\label{eq nu}
\Delta_\tau \nu_\omega - \la \nabla_\tau \nu_\omega, x\ra = -|B_E|^2\nu_\omega  + \e  \la b, \nu\ra \nu_\omega  -\e \la b, \omega\ra. 
\end{equation}

\medskip
By the divergence theorem on $\pa^* E$ we have that for any function $\varphi \in C_0^\infty(\pa^* E)$ and for any  
function $\psi \in C^1(\pa^* E)$, 
\[
\int_{\pa^* E}\div_\tau\Bigl( e^{-\frac{|x|^2}{2}}\psi\nabla_\tau \varphi \Bigr) \, d \Ha^{n-1} = 
\int_{\pa^* E} \Hh \la e^{-\frac{|x|^2}{2}}\psi\nabla_\tau \varphi,\nu^E \ra \, d \Ha^{n-1}=0.
\]
The previous equality gives us an integration by parts formula  
\[
\int_{\pa^* E} \psi (\Delta_\tau \varphi  - \la \nabla_\tau \varphi, x\ra)  \, d \Ha_\gamma^{n-1} 
=  - \int_{\pa^* E} \la \nabla_\tau \psi , \nabla_\tau \varphi \ra   \, d \Ha_\gamma^{n-1}.
\]
We will use along the paper the above formula with $\varphi = x_\omega$ or $\varphi = \nu_\omega$. Also if they do not belong to 
$C_0^\infty(\pa^* E)$, we are allowed to do so by an approximation argument (see \cite{Bar14,Ro}).

\begin{remark}
\label{eigenfunction}
We associate the following second order operator $L$  with  the first four terms in the quadratic form \eqref{second var}, 
\begin{equation}\label{operator L}
L[\varphi] := - \Delta_\tau \varphi  + \la \nabla_\tau \varphi , x\ra  -|B_E|^2 \varphi  + \e  \la b, \nu\ra \varphi -\varphi,
\end{equation}  
where $\varphi \in C_0^\infty(\partial^* E)$. By integration by parts the inequality \eqref{second var} can be written as 
\[
 \int_{\partial^* E} L[\varphi] \varphi\, d\Ha^{n-1}_\gamma 
 + \frac{\e}{\sqrt{2\pi}} \, \Big| \int_{\partial^* E} \varphi\, x \, d\Ha^{n-1}_\gamma \Big|^2 \geq 0.
\]

Note that when the vector $\omega$ is orthogonal to the barycenter, i.e.,  $\la \omega, b \ra = 0$, then by \eqref{eq nu} 
the function $\nu_\omega$ is an eigenfunction of $L$ and satisfies 
\[
L[\nu_\omega] = - \nu_\omega.
\] 
\end{remark}

For every $\unit \in \Sf^{n-1}$ it holds by the divergence theorem in $\R^n$ that  
\begin{equation*}\begin{split}
\int_{\partial^* E} \nu_\unit\, d\Ha^{n-1}_\gamma(x)
&=\frac{1}{(2\pi)^{\frac{n-1}{2}}}\int_E \div( \unit e^{-\frac{|x|^2}{2}}) \, dx\\
=&-\sqrt{2\pi}\int_E \langle x,\unit\rangle \, d\gamma(x)
=-\sqrt{2\pi}\la b, \unit \ra.
\end{split}\end{equation*}
In particular, when $\la \omega, b \ra = 0$ the function  $\varphi = \nu_\omega$ has zero average. 
Therefore by Remark \ref{eigenfunction} it is natural to use $\nu_\omega$ with $\la \omega, b \ra = 0$ as a test function in the 
second variation condition \eqref{second var}. 

The  equality $\int_{\partial^* E} \nu_\unit\, d\Ha^{n-1}_\gamma = -\sqrt{2\pi}\la b, \unit \ra$ for every  $\unit \in \Sf^{n-1}$ also implies 
\begin{equation}
\label{usein 2}
\bar{\nu} P_\gamma(E) =  - \sqrt{2 \pi} \, b.
\end{equation}
In particular, we have by  \eqref{choice epsilon up}-\eqref{bound perimeter low}
\begin{equation}\label{usein}
\frac{1}{4 s^2} |\bar{\nu}|\leq \e |b|  \leq \frac{2}{s^2} |\bar{\nu}|.
\end{equation}

We conclude this preliminary section by  providing further ``regularity'' properties from \eqref{eq x 2} for the minimizers of \eqref{functional}. 
We call the estimates in the  following lemma ``Caccioppoli inequalities'' since they 
follow from \eqref{eq x 2} by an argument which is similar to the classical proof of Caccioppoli inequality known  in elliptic PDEs. 
This result is an improved version of  \cite[Proposition 1]{Bar14}.

\begin{lemma}[\textbf{Caccioppoli inequalities}]
Let $E \subset \R^n$ be a minimizer of \eqref{functional}. 
Then for any $\omega \in  \mathbb{S}^{n-1}$ it holds
\begin{equation}\label{standard1}
\int_{\pa^* E} x_\omega^2 \, d \Ha_\gamma^{n-1} \leq (s+2)^2 \int_{\pa^* E} \nu_\omega^2 \, d \Ha_\gamma^{n-1} + 8 P_\gamma(E)
\end{equation}
and
\begin{equation}\label{standard2}
\int_{\pa^* E} (x_\omega - \bar{x}_ \omega)^2 \, d \Ha_\gamma^{n-1} \leq 
(s+2)^2 \int_{\pa^* E} (\nu_\omega - \bar{\nu}_ \omega)^2 \, d \Ha_\gamma^{n-1} +  8 P_\gamma(E).
\end{equation}

\end{lemma}
\begin{proof}
Let us first  prove \eqref{standard1}. To simplify  the notation we define   
\begin{equation*}
x_b := \begin{cases} \la x, \frac{b}{|b|} \ra   &\text{if } \, b \neq 0, \\
 0    &\text{if } \, b = 0. 
 \end{cases}
\end{equation*}
We multiply  \eqref{eq x 2} by $x_\omega$ and integrate by parts over $\pa^* E$ to get
\begin{equation}
\label{conv lem 1}
\int_{\pa^* E} x_\omega^2 \, d \Ha_\gamma^{n-1} 
=-\lambda \int_{\pa^* E} \nu_\omega x_\omega \, d \Ha_\gamma^{n-1} 
+\int_{\pa^* E} | \nabla_\tau x_\omega |^2 \, d \Ha_\gamma^{n-1} + \e |b| \int_{\pa^* E} x_b \nu_\omega  x_\omega \, d \Ha_\gamma^{n-1}
\end{equation}
We estimate the right-hand-side of \eqref{conv lem 1} in the following way.  We estimate the first term by Young's inequality 
\[
\begin{split}
-\lambda \int_{\pa^* E} \nu_\omega x_\omega \, d \Ha_\gamma^{n-1} &\leq \frac{1}{2} \int_{\pa^* E} x_\omega^2  \, d \Ha_\gamma^{n-1} 
+  \frac{\lambda^2}{2} \int_{\pa^* E} \nu_\omega^2   \, d \Ha_\gamma^{n-1} \\
&\leq \frac{1}{2} \int_{\pa^* E} x_\omega^2  \, d \Ha_\gamma^{n-1} +  \frac{(s+1)^2}{2} \int_{\pa^* E} \nu_\omega^2   \, d \Ha_\gamma^{n-1},
\end{split}
\]  
where the last inequality follows from the bound on the Lagrange multiplier 
\begin{equation}\label{bound lagrange}
|\lambda| \leq s+1
\end{equation}
given by Theorem \ref{thm reg} and by our choice of $\Lambda$ in \eqref{choice lambda}. 
Since $|\nabla_\tau x_\omega |^2 =  1-\nu_\omega^2 \leq 1$, we may bound the second term simply by 
\[
\int_{\pa^* E} | \nabla_\tau x_\omega |^2 \, d \Ha_\gamma^{n-1} \leq P_\gamma(E).
\]
Finally we bound the last term again by Young's inequality and by  $\e |b| \leq \frac{2}{s^2}$ (proved in \eqref{usein})
\[
 \e |b| \int_{\pa^* E} x_b \nu_\omega  x_\omega \, d \Ha_\gamma^{n-1} 
 \leq \frac{1}{s^2} \int_{\pa^* E} x_\omega^2 \, d \Ha_\gamma^{n-1} 
 + \frac{1}{s^2} \int_{\pa^* E} x_b^2 \, d \Ha_\gamma^{n-1}. 
\]
By using these three estimates in \eqref{conv lem 1} we obtain  
\begin{equation}
\label{conv lem 1 1}
\left(\frac{1}{2}- \frac{1}{s^2} \right)\int_{\pa^* E} x_\omega^2 \, d \Ha_\gamma^{n-1}  
\leq \frac{(s+1)^2}{2} \int_{\pa^* E} \nu_\omega^2 \, d \Ha_\gamma^{n-1} 
+ P_\gamma(E)  + \frac{1}{s^2} \int_{\pa^* E} x_b^2 \, d \Ha_\gamma^{n-1}.
\end{equation}

If the barycenter is zero the claim follows immediately from \eqref{conv lem 1 1}. 
If $b \neq 0$, we first use \eqref{conv lem 1 1} with $\omega = \frac{b}{|b|}$ and obtain
\[
\begin{split}
\left(\frac{1}{2}- \frac{2}{s^2} \right)\int_{\pa^* E} x_b^2 \, d \Ha_\gamma^{n-1}  
&\leq \frac{(s+1)^2}{2} \int_{\pa^* E} \nu_b^2 \, d \Ha_\gamma^{n-1} + P_\gamma(E) \\
&\leq \left(\frac{(s+1)^2}{2} +1 \right) P_\gamma(E).
\end{split}
\]
This implies
\begin{equation}
\label{conv lem 1 3}
\int_{\pa^* E} x_b^2 \, d \Ha_\gamma^{n-1}  \leq 2s^2 P_\gamma(E).
\end{equation}
Therefore we have by \eqref{conv lem 1 1} 
\[
\left(\frac{1}{2}- \frac{1}{s^2} \right)\int_{\pa^* E} x_\omega^2 \, d \Ha_\gamma^{n-1}  
\leq \frac{(s+1)^2}{2} \int_{\pa^* E} \nu_\omega^2 \, d \Ha_\gamma^{n-1} + 3 P_\gamma(E) .
\]
This yields the claim.

\medskip
The proof of  the second inequality is similar. We multiply the equation \eqref{eq x 2} by $(x_\omega - \bar{x}_ \omega)$ 
and integrate by parts over $\pa^* E$ to get
\[
\begin{split}
\int_{\pa^* E} (x_\omega - \bar{x}_ \omega)^2 \, d \Ha_\gamma^{n-1} 
=&-\lambda \int_{\pa^* E} (x_\omega - \bar{x}_ \omega)(\nu_\omega - \bar{\nu}_ \omega) \, d \Ha_\gamma^{n-1} 
+\int_{\pa^* E} | \nabla_\tau x_\omega |^2 \, d \Ha_\gamma^{n-1} \\
&+\e |b| \int_{\pa^* E} x_b \nu_\omega  (x_\omega - \bar{x}_ \omega) \, d \Ha_\gamma^{n-1}.
\end{split}
\]
By estimating the three terms on the right-hand-side precisely as before, we deduce
\[
\begin{split}
\left(\frac{1}{2}- \frac{1}{s^2} \right)\int_{\pa^* E} (x_\omega - \bar{x}_ \omega)^2 \, d \Ha_\gamma^{n-1}  
&\leq \frac{(s+1)^2}{2} \int_{\pa^* E} (\nu_\omega - \bar{\nu}_ \omega) ^2 \, d \Ha_\gamma^{n-1} + P_\gamma(E)  
+ \frac{1}{s^2} \int_{\pa^* E} x_b^2 \, d \Ha_\gamma^{n-1}\\
&\leq \frac{(s+1)^2}{2} \int_{\pa^* E} (\nu_\omega - \bar{\nu}_ \omega) ^2 \, d \Ha_\gamma^{n-1} + 3 P_\gamma(E),
\end{split}
\]
where the last inequality follows from \eqref{conv lem 1 3}. This implies \eqref{standard2}. 
\end{proof}

%%%%%%%%%%%%%%%%%%%%%%%%%%%%%%%%%%%%%%%%%%%%%%%%%%%%%%%%%%%%%%%%%%%%%%%%%%%%%%%%%%%%%%%%%%%%%%%%%%%%%%%%%%%%%%%%%%%%%%
\section{Reduction to the two dimensional case}\label{section redux n to 2}

\noindent
In this section we prove that it is enough to obtain the result in the two dimensional case.
More precisely, we  prove the following result.

\begin{theorem}\label{n to 2}
Let $E$ be a minimizer of \eqref{functional}. Then, up to a rotation, $E=F\times\R^{n-2}$ for some set $F\subset\R^2$. 
\end{theorem}
\begin{proof}
Let  $\{e^{(1)},\ldots,e^{(n)}\}$ be an orthonormal basis  of $\R^n$.  We begin with a simple observation: if $i \neq j$ then 
by the  divergence theorem 
\[
\int_{\partial^* E} x_i \nu_j \, d \Ha_\gamma^{n-1} = -\sqrt{2 \pi} \int_{E} x_ix_j \, d\gamma. 
\]
In particular, the matrix $A_{ij} = \int_{\partial E} x_i \nu_j \, d \Ha_\gamma^{n-1}$ is symmetric. We may therefore assume 
that $A_{ij}$ is diagonal, by changing the basis of $\R^n$ if necessary. In particular, it holds
\begin{equation}\label{null out diagonal}
\int_{\partial^* E} x_i \nu_j \, d \Ha_\gamma^{n-1} = 0 \qquad \text{for } \, i \neq j.
\end{equation}
By reordering the elements of the basis we may also assume that 
\begin{equation}\label{order}
\int_{\partial^* E}x_j^2 \,d \Ha_\gamma^{n-1}\geq\int_{\partial^* E}x_{j+1}^2 \,d\Ha_\gamma^{n-1}
\end{equation}
for $j\in\{1,\dots,n-1\}$.

Since we assume $n\geq3$, we may choose a direction $\unit\in\Sf^{n-1}$ which is  orthogonal both to the barycenter $b$ and to $e^{(1)}$. 
To be more precise, we choose $\unit$ such that $\la \omega,b\ra=0$ and $\unit \in \mathrm{span}\{e^{(2)},e^{(3)}\}$.  
Since $\la \omega,b\ra=0$, \eqref{usein 2} yields $\bar{\nu}_\unit = 0$. 
In other words, the  function $\nu_\unit$ has zero average. We use $\varphi=\nu_\unit$ as a test function in the second variation 
condition \eqref{second var}. According to Remark \ref{eigenfunction} we may write the inequality \eqref{second var} as 
\[
 \int_{\partial^* E} L[\nu_\unit] \nu_\unit\, d\Ha^{n-1}_\gamma 
 + \frac{\e}{\sqrt{2\pi}} \, \Big| \int_{\partial^* E} \nu_\unit\, x \, d\Ha^{n-1}_\gamma \Big|^2 \geq 0,
\]
where the operator $L$ is defined in \eqref{operator L}. Since $\unit$ is  orthogonal to $b$ we deduce by Remark \ref{eigenfunction} 
that $\nu_\unit$ is an eigenfunction of $L$ and satisfies $L[\nu_\unit] = - \nu _\unit$. Therefore we get 
\begin{equation}\label{ineq1}
-\int_{\pa^* E} \nu_\unit^2 \, d \Ha_\gamma^{n-1}
+\frac{\e}{\sqrt{2\pi}} \, \Big| \int_{\partial^* E} \nu_\unit\, x \, d\Ha_\gamma^{n-1} \Big|^2\geq0.
\end{equation}

The crucial step in the proof is  to estimate the second term in \eqref{ineq1}, by showing that it is small enough. 
This is possible, because  $\unit$ is orthogonal to $e^{(1)}$. Indeed, by using \eqref{null out diagonal} and the 
fact that $\unit\in\mathrm{span}\{e^{(2)},e^{(3)}\}$, and then Cauchy-Schwarz inequality, we get
\begin{equation*}\begin{split}
\Bigl| \int_{\partial^* E} \nu_\unit\,x\, d \Ha_\gamma^{n-1} \Bigr|^2 
=&\Bigl( \int_{\partial^* E} x_2 \nu_\unit \, d \Ha_\gamma^{n-1} \Bigr)^2 
+ \Bigl( \int_{\partial^* E} x_3 \nu_\unit \, d \Ha_\gamma^{n-1} \Bigr)^2\\
\leq& \Bigl(\int_{\partial^* E} x_2^2+x_3^2\, d \Ha_\gamma^{n-1} \Bigr)
\Bigl(\int_{\partial^* E} \nu_\unit^2\, d \Ha_\gamma^{n-1} \Bigr).
\end{split}\end{equation*}
We estimate the first term on the right-hand-side first  by  \eqref{order},  then by the Caccioppoli estimate \eqref{standard1} 
and finally by  \eqref{bound perimeter up}
\begin{equation}\label{key1}\begin{split}
\int_{\partial^* E}(x_2^2+x_3^2) \, \Ha_\gamma^{n-1}
&\leq \frac{2}{3}\int_{\partial^* E}(x_1^2+x_2^2+x_3^2) \, \Ha_\gamma^{n-1}\\
&\leq \frac{2}{3}\biggl[(s+2)^2\int_{\partial^* E}(\nu_1^2+\nu_2^2+\nu_3^2) \, \Ha_\gamma^{n-1}+24 P_\gamma(E)\biggr]\\
&\leq \frac{2}{3}\bigl[(s+2)^2+24\bigr] P_\gamma(E)
\leq \frac{9}{13}s^2e^{-\frac{s^2}{2}}.
\end{split}\end{equation}
Since we assume  $\e \leq\frac{7\sqrt{2 \pi}}{5s^2}  e^{\frac{s^2}{2}}$ (see \eqref{choice epsilon up}), the previous two inequalities yield
\begin{equation}\label{ineq2}
\frac{\e}{\sqrt{2\pi}} \, \Bigl| \int_{\partial^* E} \nu_\unit\,x\, d \Ha_\gamma^{n-1} \Bigr|^2 
\leq\frac{63}{65}\int_{\partial^* E} \nu_\unit^2\, d \Ha_\gamma^{n-1}.
\end{equation}
Then, by collecting \eqref{ineq1} and \eqref{ineq2} we obtain
\begin{equation}\label{finale}
-\int_{\partial^* E}\nu_\unit^2 \, d \Ha_\gamma^{n-1} \geq 0.
\end{equation}
This implies $\nu_\unit =  0$.
% Since $E$ has locally finite perimeter in $\R^n$, De Giorgi's structure theorem
%  \cite[Theorem 15.9]{Ma} yields
% \begin{equation*}
% D\chi_E=-\nu\Ha^{n-1}\lfloor\partial^*E.
% \end{equation*}
% Therefore, the distributional partial derivatives in the direction $\unit$
% is null and necessarily the function $\chi_E$ is constant in that direction.
% Lef $F$ be the intersection of $E$ with the hyperplane $H$ perpendicular to $\unit$.
% We have $b(E)=b(F)$ and also $\Ha^{n-1}_\gamma(\partial^*E)=\Ha^{n-2}_\gamma((\partial^*E)\cap H)$.
% Since the essential boundary of $F$ in $H$ coincides with $(\partial^*E)\cap H$ (see \cite{Volp}) 
We have thus reduced the problem from $n$ to $n-1$. By repeating the previous argument we reduce
the problem to the planar case. 
\end{proof}

\begin{remark} \label{why not}
We have to be careful in our  choice of direction $\unit$, and in general we may not simply choose any direction orthogonal to the barycenter $b$.
Indeed, if $\unit, v \in \Sf^{n-1}$ are  vectors such that $\la b, \unit \ra = 0$ and  
\begin{equation}\label{rewrite}
\Big| \int_{\partial^* E}  \nu_\unit \,x \, d\Ha^{n-1}_\gamma\Big|  
= \la \int_{\partial^* E} \nu_\unit\, x\, d\Ha^{n-1}_\gamma, \upsilon\ra
=  \int_{\partial^* E} \nu_\unit \la x, \upsilon \ra  \, d\Ha^{n-1}_\gamma.
\end{equation}
Then, by using  Cauchy-Schwarz inequality, we may estimate the second term in \eqref{ineq1} by
\begin{equation*}
\frac{\e}{\sqrt{2\pi}} \, \Big| \int_{\partial^* E} \nu_\unit\, x \, d\Ha^{n-1}_\gamma \Big|^2
\leq\frac{\e}{\sqrt{2\pi}} \left( \int_{\partial^* E} x_\upsilon^2 \, d\Ha^{n-1}_\gamma\right)
\left( \int_{\partial^* E} \nu_\unit^2 \, d\Ha^{n-1}_\gamma\right).
\end{equation*}
We may estimate the term $\frac{\e}{\sqrt{2\pi}} \int_{\partial^* E} x_\upsilon^2 \, d\Ha^{n-1}_\gamma$ by the Caccioppoli 
estimate  \eqref{standard1}, and by \eqref{bound perimeter up} and \eqref{choice epsilon up}  
\begin{equation*}
\frac{\e}{\sqrt{2\pi}} \, \Bigl| \int_{\partial^* E} \nu_\unit\,x\, d \Ha_\gamma^{n-1} \Bigr|^2 
\leq\frac{8}{5}\int_{\partial^* E} \nu_\unit^2\, d \Ha_\gamma^{n-1}
\end{equation*}
instead of \eqref{ineq2}. Unfortunately this estimate is not good enough.
Note that we cannot shrink $\e$, since we have the constrain given by \eqref{choice epsilon down}.
\end{remark}

\begin{remark}
\label{rem b=0}
We may  further reduce the problem to the one dimensional case
if $b=0$, since we may use $\unit=e^{(2)}$ in the previous argument ($\nu_\unit$ has zero average and 
$\int_{\partial^* E}x_2^2$ is small enough).
However, this is a special case and a priori nothing guaranties that $b=0$.
Because of that we have to handle the reduction to the one dimensional case in a different way.
\end{remark}

%%%%%%%%%%%%%%%%%%%%%%%%%%%%%%%%%%%%%%%%%%%%%%%%%%%%%%%%%%%%%%%%%%%%%%%%%%%%%%%%%%%%%%%%%%%%%%%%%%%%%%%%%%%
\section{Reduction to the one dimensional case}\label{section redux 2 to 1}

\noindent
In this section we will prove a further reduction of the problem, by showing that it is enough to obtain the result in the one 
dimensional case. This is technically more involved than  Theorem~\ref{n to 2} and requires more a priori  information on the minimizers.

\begin{theorem}\label{2 to 1}
Let $E$ be a minimizer of \eqref{functional}. Then, up to a rotation, $E=F\times\R^{n-1}$ for some set $F\subset\R$.
\end{theorem}

Thanks to Theorem \ref{n to 2} we may assume from now on that $n=2$. 
In particular, by Theorem \ref{thm reg} the boundary is regular and $\partial E=\partial^* E$.
Moreover the Euler equation and \eqref{eq nu} simply read as
\begin{equation}
\label{euler 2d}
k = \lambda + \la x, \nu\ra  -  \e \la b,x\ra,
\end{equation}
\begin{equation}
\label{eq nu 2}
\Delta_\tau \nu_\omega - \la \nabla_\tau \nu_\omega, x\ra = -k^2\nu_\omega  + \e  \la b, \nu\ra \nu_\omega  -\e \la b, \omega\ra,
\end{equation}
where $k$ is the curvature of $\partial E$. 

The idea is to proceed by using the second variation argument once more, but this time  in a direction
that it is not necessarily orthogonal to the barycenter. This argument does not reduce the problem to $\R$, but  
gives us the following information on the minimizers. 

\begin{lemma}\label{key lemma}
Let $E \subset \R^2$ be a minimizer of  \eqref{functional}. Then 
\begin{equation}\label{key lemma1}
\fint_{\pa E}  k^2 \, d \Ha_\gamma^1 \leq \frac{2}{s^2}.
\end{equation}
Moreover, there exists a direction $v \in \Sf^1$ such that
\begin{equation}\label{key lemma2}
\fint_{\pa E} (\nu_v - \bar{\nu}_v)^2 \, d \Ha_\gamma^1  
\leq \frac{10}{s^2} \, \bar{\nu}_v^2.
\end{equation}
\end{lemma} 

Observe that the above estimate implies that $\nu_v$ is close to a constant.
In particular, this excludes the minimizers to be close to the disk. 

\begin{proof}
We begin by showing that for any $\omega \in  \mathbb{S}^1$ it holds
\begin{equation}\label{ineq0}\begin{split}
\bar{\nu}_\unit^2\int_{\pa E}& k^2 \, d \Ha_\gamma^1 
+\int_{\pa E} |\nu_\unit-\bar{\nu}_\unit|^2 \, d \Ha_\gamma^1\\
&\leq\frac{2}{s^2} \, \bar{\nu}_\unit^2 P_\gamma(E)
+ \frac{\e}{\sqrt{2\pi}} \, \Big| \int_{\pa E} (\nu_\unit-\bar{\nu}_\unit)\, x \, d\Ha_\gamma^1 \Big|^2.
\end{split}\end{equation}
To this aim we choose  $\varphi=\nu_\unit-\bar{\nu}_\unit$  as a test function in the second variation condition \eqref{second var}.
We remark that because $\unit$ is not  in general orthogonal to the barycenter $b$,  neither $\nu_\unit$ or  $\nu_\unit-\bar{\nu}_\unit$ 
is an eigenfunction of the operator $L$ associated with the second variation defined in Remark \ref{eigenfunction}.   

We multiply the equation \eqref{eq nu 2}  by $\nu_\unit$ and integrate by parts to obtain
\begin{equation}\label{by part1}
\int_{\pa E} \left(|\nabla_\tau \nu_\omega|^2 -k^2\nu_\unit^2 + \e \la b, \nu \ra\nu_\unit^2\right)\, d \Ha_\gamma^1  
=\e \la b, \unit\ra  \bar{\nu}_\unit P_\gamma(E), 
\end{equation}
and simply integrate \eqref{eq nu} over $\pa E$ to get 
\begin{equation}\label{by part2}
\int_{\pa E} \left(k^2\nu_\unit - \e \la b, \nu \ra\nu_\unit\right) \, d \Ha_\gamma^1
=- \e  \la b, \unit\ra P_\gamma(E).
\end{equation}
Hence, by also using $\bar{\nu} P_\gamma(E) =  - \sqrt{2 \pi} \, b$ (see \eqref{usein 2}),  we may write  
\begin{equation*}
\begin{split}
\int_{\pa E}& \left(|\nabla_\tau \nu_\omega|^2 -k^2(\nu_\unit-\bar{\nu}_\unit)^2 
+ \e \la b, \nu \ra(\nu_\unit-\bar{\nu}_\unit)^2\right)\, d \Ha_\gamma^1\\
&=-\bar{\nu}_\unit^2\int_{\pa E}k^2\, d \Ha_\gamma^1
+\e\la b,\bar{\nu}\ra\bar{\nu}_\unit^2P_\gamma(E)-\e \la b, \unit \ra \,  \bar{\nu}_\unit P_\gamma(E)\\
&=-\bar{\nu}_\unit^2\int_{\pa E}k^2\, d \Ha_\gamma^1
+\frac{\e}{\sqrt{2\pi}}\, (1-|\bar{\nu}|^2)\bar{\nu}_\unit^2 P^2_\gamma(E)\\
&\leq-\bar{\nu}_\unit^2\int_{\pa E}k^2\, d \Ha_\gamma^1
+\frac{2}{s^2}\bar{\nu}_\unit^2 P_\gamma(E),
\end{split}
\end{equation*}
where in the last inequality we have used  the estimates  \eqref{choice epsilon up} and \eqref{bound perimeter up}. 
The above inequality and the second variation condition \eqref{second var} with  $\varphi=\nu_\unit-\bar{\nu}_\unit$ imply  \eqref{ineq0}.

Let us consider an orthonormal basis $\{e^{(1)}, e^{(2)}\}$ of $\R^2$ and assume 
$\int_{\partial E}x_1^2 \,d \Ha_\gamma^1\geq\int_{\partial E}x_2^2 \,d\Ha_\gamma^1$.  
As in  \eqref{key1}, we  use  the Caccioppoli estimate \eqref{standard1} and  \eqref{bound perimeter up} to  get
\begin{equation}\begin{split}\label{x hat 2}
\int_{\pa E}  x_2^2 \, d \Ha_\gamma^1 
&\leq \frac{1}{2}\int_{\partial E}(x_1^2+x_2^2) \, \Ha_\gamma^1\\
&\leq \frac{1}{2}\bigl[(s+2)^2+16\bigr] P_\gamma(E)
\leq \frac{1}{2}\bigl[(s+4)^2\bigr]e^{-\frac{s^2}{2}}.
\end{split}\end{equation}
We choose  a direction $v \in\Sf^1$ which is orthogonal to the vector
$\int_{\partial E}x_1(\nu-\bar{\nu}) \,d \Ha_\gamma^1$.
Since $\fint_{\pa E} x_1 (\nu_v- \bar{\nu}_v) \, d \Ha_\gamma^1= \la\fint_{\pa E} x_1 (\nu- \bar{\nu}) \, d \Ha_\gamma^1,v\ra=0$, we have
\begin{equation*}
 \Big| \fint_{\pa E} x (\nu_v- \bar{\nu}_v) \, d \Ha_\gamma^1 \Big| 
 = \Big| \fint_{\pa E} x_2 (\nu_v- \bar{\nu}_v) \, d \Ha_\gamma^1 \Big|.
\end{equation*}
Then, by the above equality, by Cauchy-Schwarz inequality and by  \eqref{x hat 2} we have 
\begin{equation*}\begin{split}
 \Big| \int_{\partial E} x (\nu_v - \bar{\nu}_v)  \, d \Ha_\gamma^{1} \Big|^2 
 &= \left(  \int_{\partial E} x_2 (\nu_v - \bar{\nu}_v)  \, d \Ha_\gamma^{1} \right)^2\\
&\leq   \left(  \int_{\partial E} x_2^2\, d \Ha_\gamma^{1} \right)  
\left(  \int_{\partial E}  (\nu_v - \bar{\nu}_v)^2  \, d \Ha_\gamma^{1} \right)\\
&\leq   \frac{(s+4)^2}{2} e^{-\frac{s^2}{2}}  \left(  \int_{\partial E}  (\nu_v - \bar{\nu}_v)^2  \, d \Ha_\gamma^{1} \right).
\end{split}\end{equation*}
With the bound  $\e \leq\frac{7\sqrt{2 \pi}}{5s^2}  e^{\frac{s^2}{2}}$   (see \eqref{choice epsilon up}), the previous inequality yields
\begin{equation*}
\frac{\e}{\sqrt{2\pi}} \, \Bigl| \int_{\partial E} (\nu_v-\bar{\nu}_v)\,x\, d \Ha_\gamma^1 \Bigr|^2 
\leq\frac{4}{5}\int_{\partial E} (\nu_v-\bar{\nu}_v)^2\, d \Ha_\gamma^1.
\end{equation*}
Hence, the   inequality \eqref{ineq0} implies 
\begin{equation*}
\bar{\nu}_v^2\int_{\partial E} k^2  \, d \Ha_\gamma^{1} 
+ \frac{1}{5}\int_{\partial E} (\nu_v - \bar{\nu}_v)^2  \, d \Ha_\gamma^{1}  
\leq\frac{2}{s^2} \, \bar{\nu}_v^2 P_\gamma(E).
\end{equation*}
From this inequality we have immediately \eqref{key lemma2}, and also \eqref{key lemma1}, if $\bar{\nu}_v$ is not zero.
If instead $\bar{\nu}_v=0$, then also $\nu_v = 0$ by  \eqref{key lemma2}. Thus  $\pa E$ is flat, $k = 0$  and \eqref{key lemma1} holds again.
\end{proof}

\medskip
We will also need the following auxiliary result.   
\begin{lemma} 
Let $E \subset \R^2$ be a minimizer of  \eqref{functional}. Then, for every $x \in \pa E$ it holds 
\begin{equation}\label{aux 1}
|x| \geq s - 1.
\end{equation} 
\end{lemma}

\begin{proof}
We argue by contradiction and assume that there exists $\tilde{x} \in \pa E $ such that $|\tilde{x}| < s - 1$.  We claim that then it holds 
\begin{equation}
\label{aux lem 1 1}
\Ha^1(\pa E \cap B_{1/2}(\tilde{x})) \geq \frac{1}{s} .
\end{equation}
We remark that $\Ha^1$ is the standard Hausdorff measure, i.e., $\Ha^1(\pa E \cap B_{1/2}(\tilde{x}))$ denotes the length of the curve.  
We divide the proof of \eqref{aux lem 1 1} in two cases.

Assume first that there is a component  of $\pa E$, say $\tilde{\Gamma}$, which is  contained in the disk $B_{1/2}(\tilde{x})$. 
By regularity, $\tilde{\Gamma}$ is a smooth Jordan curve which encloses a bounded set $\tilde{E}$, i.e., $\tilde{\Gamma} = \pa \tilde{E}$. 
Note that then it holds $\tilde{E} \subset B_{R}$ for $ R= s - 1/2$. We integrate the Euler equation \eqref{euler 2d} over  $\pa \tilde{E}$  
with respect to the standard Hausdorff measure and obtain  by the Gauss-Bonnet formula and by the divergence theorem that 
\begin{equation} \label{lem aux 1 0}
\begin{split}
2 \pi  = \int_{\tilde{\Gamma}} k \, d \Ha^1 &=  \int_{\tilde{\Gamma}} \left( \la x, \nu \ra  + \lambda  - \e \la b, x \ra \right) \, d \Ha^1 \\
&\leq 2 |\tilde{E}| + \left(|\lambda| + \frac{2}{s}\right) \Ha^1(\tilde{\Gamma}),
\end{split}
\end{equation}
where in the last inequality we have used $\e |b| \leq \frac{2}{s^2}$ (proved in \eqref{usein})  and the fact that  for all $x \in \tilde{E}$ 
it holds $|x| \leq s - 1/2$.  The isoperimetric inequality in $\R^2$ implies
\[
|\tilde{E}|  \leq \frac{1}{4 \pi^2}\Ha^1(\tilde{\Gamma})^2.
\]
Therefore since  $|\lambda| \leq s + 1$  we obtain  from \eqref{lem aux 1 0} that
\[
2 \pi  \leq  \frac{1}{2 \pi^2} \Ha^1(\tilde{\Gamma})^2 + (s+2)  \Ha^1(\tilde{\Gamma}).
\]
This implies $\Ha^1(\tilde{\Gamma}) \geq \frac{1}{s}$ and the claim \eqref{aux lem 1 1}  follows.

Let us then assume that  no component of $\pa E$ is contained in $B_{1/2}(\tilde{x})$. In this case the boundary curve  
passes $\tilde{x}$ and exists the disk $B(\tilde{x},\frac12)$.  In particular,  it holds  
$\Ha^1(\pa E \cap B_{1/2}(\tilde{x}))\geq 1/2$ which implies \eqref{aux lem 1 1}.

Since  for all $x \in \pa E \cap B_{1/2}(\tilde{x})$ it holds $|x| \leq s -1/2$, the estimate \eqref{aux lem 1 1}  implies     
\[
\begin{split}
P_\gamma(E) &\geq \frac{1}{\sqrt{2\pi}} \int_{\pa E \cap B_{1/2}(\tilde{x})} e^{-\frac{|x|^2}{2}} \, d \Ha^1  \\
&\geq \frac{1}{\sqrt{2\pi} } e^{-\frac{(s -1/2)^2}{2}} \Ha^1(\pa E \cap B_{1/2}(\tilde{x})) 
\geq 2  e^{-\frac{s^2}{2}}.
\end{split}
\]
This contradicts \eqref{bound perimeter up}.
\end{proof}

\medskip
For the remaining part of this section we choose a basis $\{e^{(1)},e^{(2)}\}$ for $\R^2$ such that $e^{(1)} = v$, 
where $v$ is the direction  in Lemma~\ref{key lemma} and $e^{(2)}$ is an orthogonal direction to that. 
The disadvantage of Lemma \ref{key lemma} is that the argument does not seem to give us any  information on $\nu_2 = \la\nu,e^{(2)} \ra$. 
However, by studying closely the proof of Lemma \ref{key lemma} we may reduce to the case when it holds
\begin{equation}
\label{nu 2 big}
\fint_{\pa E} (\nu_2 -\bar{\nu}_2)^2 \, d \Ha_\gamma^1 \geq \frac47.
\end{equation}
Indeed, we conclude below that if \eqref{nu 2 big} does not hold then  the argument of the proof of  Lemma~\ref{key lemma} 
yields that the minimizer is one-dimensional. 
In fact, by the one dimensional analysis in Section~\ref{section 1d} we deduce that if \eqref{nu 2 big} does not hold then the minimizer is the half-space. 

To  show  \eqref{nu 2 big}, we argue by contradiction, in which case it holds 
$\fint_{\pa E}  (\nu_2 -\bar{\nu}_2)^2 \, d \Ha_\gamma^1 < \frac{4}{7}$. 
Then the Caccioppoli estimate \eqref{standard2} yields
\[
\fint_{\pa E} (x_2- \bar{x}_2)^2 \, d \Ha_\gamma^1  \leq \frac47(s+2)^2 + 8
\]
while again by \eqref{standard2} and by \eqref{key lemma2} from Lemma \ref{key lemma} we have
\[
\fint_{\pa E} (x_1-\bar{x}_1)^2 \, d \Ha_\gamma^1  \leq (s+2)^2\fint_{\pa E} (\nu_1-\bar{\nu}_1)^2 \, d \Ha_\gamma^1 + 8 \leq C. 
\]
Let  $e \in \Sf^1$ be orthogonal to the barycenter $b$. We now apply the argument in the proof of Lemma \ref{key lemma} for the 
test function $\varphi = \nu_e$. By the two above inequalities and by \eqref{bound perimeter up} we have 
\[
\begin{split}
\Bigl| \int_{\partial E} x\, \nu_e \, d \Ha_\gamma^{1} \Bigr|^2 
=&\Bigl( \int_{\partial E}  (x_1-\bar{x}_1) \nu_e \, d \Ha_\gamma^{1} \Bigr)^2 
+ \Bigl( \int_{\partial E} (x_2- \bar{x}_2) \nu_e \, d \Ha_\gamma^{1} \Bigr)^2\\
\leq& \left(\int_{\partial E}  (x_1-\bar{x}_1)^2 +(x_2- \bar{x}_2)^2 \, d \Ha_\gamma^{1} \right) 
\int_{\partial E} \nu_e^2\, d \Ha_\gamma^{1}\\
\leq&\frac{9}{13}s^2e^{-\frac{s^2}{2}}\int_{\partial E} \nu_e^2\, d \Ha_\gamma^{1}.
\end{split}
\]
In other words,  since $\e \leq\frac{7\sqrt{2 \pi}}{5s^2}  e^{\frac{s^2}{2}}$ we conclude that  the crucial 
estimate \eqref{ineq2} in the proof of Theorem \ref{n to 2} holds for  
a direction orthogonal to the barycenter and thus by Remark \ref{rem b=0} 
we conclude that $\nu_e = 0$. Hence, we may assume from now on that  \eqref{nu 2 big} holds.

\medskip

Let us  define
\[
\Sigma_+ = \{ x \in \pa E : x_2 > 0\} \qquad \text{and} \qquad \Sigma_- = \{ x \in \pa E : x_2 < 0\}. 
\]
 In the next lemma we use \eqref{key lemma2} from Lemma \ref{key lemma} and \eqref{nu 2 big}  to  conclude first  
 that $\Sigma_+$ and $\Sigma_-$ are flat in shape. The second estimate in the next lemma states  roughly speaking that the 
 Gaussian measure of $\{x \in \pa E : |x_2| \leq \tfrac{s}{3} \}$  is small. The latter estimate implies that, from measure 
 point of view, $\Sigma_+$ and $\Sigma_-$ are almost disconnected. 
This enables us to  variate  $\Sigma_+$ and $\Sigma_-$ separately, which will be crucial in the proof of Theorem \ref{2 to 1}. 
Recall that, given a function $f : \pa E \to \R$, we denote $(f)_{\Sigma_+} = \fint_{\Sigma_+} f\, d \Ha_\gamma^1$ 
and $(f)_{\Sigma_-} :=  \fint_{\Sigma_-} f\, d \Ha_\gamma^1$.

\begin{lemma}\label{extra info}
Let $E \subset \R^2$ be a minimizer of \eqref{functional} and assume \eqref{nu 2 big} holds. Then we have the following: 
\begin{equation}\label{extra2}
\int_{\Sigma_\pm} (x_i - (x_i)_{\Sigma_\pm})^2 \leq CP_\gamma(E) , \qquad \text{for } i = 1,2
\end{equation}
\begin{equation}\label{extra1}
\int_{\pa E \cap \{ |x_2| \leq  \tfrac{s}{3} \}}|x|^2 \, d \Ha_\gamma^1 \leq C\int_{\pa E} \nu_1^2 \, d \Ha_\gamma^1
\end{equation}
\end{lemma}

\begin{proof} \textbf{Inequality \eqref{extra2}.}
We first observe that the claim \eqref{extra2} is almost trivial for $i =1$. Indeed, by the Caccioppoli estimate \eqref{standard2}
and by \eqref{key lemma2} from Lemma \ref{key lemma} (recall that we have chosen $e^{(1)}= v$) we have  
\[
\begin{split}
\int_{\Sigma_+} (x_1 - (x_1)_{\Sigma_+})^2 \, d \Ha_\gamma^1&\leq  \int_{\Sigma_+} (x_1 - \bar{x}_1)^2 \, d \Ha_\gamma^1 \\
&\leq \int_{\pa E} (x_1 - \bar{x}_1)^2 \, d \Ha_\gamma^1 \\
&\leq  (s+2)^2 \int_{\pa E} (\nu_1 - \bar{\nu}_1)^2 \, d \Ha_\gamma^1 + 8 P_\gamma(E)\\
&\leq CP_\gamma(E).
\end{split}
\]
Thus we need to prove \eqref{extra2} for $i =2$.

We first show that 
\begin{equation}
\label{extra 2 1}
\int_{\Sigma_+} \left(|\nu_2| - (|\nu_2|)_{\Sigma_+}\right)^2 \, d \Ha_\gamma^1 \leq \frac{C}{s^2} P_\gamma(E).
\end{equation}
Note  that  \eqref{nu 2 big}  implies $\fint_{\pa E} \nu_2^2  \, d \Ha_\gamma^1  \geq \frac47$. By Jensen's inequality  we then have 
\begin{equation}
\label{extra 2 ???}
 \bar{\nu}_1^2 \leq \fint_{\pa E} \nu_1^2 \, d \Ha_\gamma^1 \leq \frac{3}{7}.
\end{equation}
Therefore  we deduce by \eqref{extra 2 ???} and  by \eqref{key lemma2}  
\[
\begin{split}
\int_{\Sigma_+} \left(|\nu_2| - \sqrt{1 - \bar{\nu}_1^2}\right)^2 \, d \Ha_\gamma^1 
&=  \int_{\Sigma_+} \frac{\left(\nu_2^2 - (1 - \bar{\nu}_1^2)\right)^2}{\big(|\nu_2| + \sqrt{1 - \bar{\nu}_1^2}\, \big)^2 } \, d \Ha_\gamma^1 \\
&\leq 2 \int_{\pa E} (\nu_1^2 -  \bar{\nu}_1^2)^2 \, d \Ha_\gamma^1 \\
&\leq 8 \int_{\pa E} (\nu_1 -  \bar{\nu}_1)^2 \, d \Ha_\gamma^1 \\
&\leq \frac{C}{s^2} P_\gamma(E).
\end{split}
\]
Since 
\[
\int_{\Sigma_+} \left(|\nu_2| - (|\nu_2|)_{\Sigma_+}\right)^2 \, d \Ha_\gamma^1 
\leq \int_{\Sigma_+} \left(|\nu_2| - \sqrt{1 - \bar{\nu}_1^2}\right)^2 \, d \Ha_\gamma^1 
\]
we have \eqref{extra 2 1}. 

To prove the inequality \eqref{extra2}  for $i=2$ we multiply the equation \eqref{eq x 2}, with $\unit = e_2$, 
by $(x_2+\lambda \nu_2)$ and integrate by parts 
\[
\int_{\pa E} (x_2  + \lambda \nu_2)^2 \, d \Ha_\gamma^1 \leq \int_{\pa E} 
\la \nabla_\tau (x_2 + \lambda \nu_2), \nabla_\tau x_2 \ra - \e \la b, x \ra \nu_2(x_2 + \lambda \nu_2) \, d \Ha_\gamma^1.
\]
We estimate the first term on the right-hand-side by Young's inequality and  by $|\lambda| \leq s+1$
\[
\la \nabla_\tau (x_2 + \lambda \nu_2), \nabla_\tau x_2 \ra \leq 2|\nabla_\tau x_2|^2 + \lambda^2 |\nabla_\tau \nu_2|^2 \leq 2 + (s +1)^2k^2
\]
and the second as
\[
\e \la b, x \ra \nu_2(x_2 + \lambda \nu_2) \leq 2 \e |b| \left( |x|^2 + (s +1)^2 \nu_2^2 \right).
\]
Hence, we have by $\e |b|\leq \frac{2}{s^2}$ (proved in \eqref{usein}), \eqref{standard2} and \eqref{key lemma1}  that 
\[
\begin{split}
\int_{\pa E} (x_2  + \lambda \nu_2)^2 \, d \Ha_\gamma^1 &\leq \int_{\pa E} \left(  2+ (s +1)^2 k^2 
+ \frac{4}{s^2}( |x|^2 + (s +1)^2\nu_2^2)\right)  \, d \Ha_\gamma^1 \\
&\leq C P_\gamma(E).
\end{split}
\]  
Therefore  it holds (recall that $x_2 >0$ on $\Sigma_+$)
\[
\begin{split}
C P_\gamma(E) &\geq \int_{\pa E} (x_2  + \lambda \nu_2)^2 \, d \Ha_\gamma^1 \geq \int_{\Sigma_+} (x_2  + \lambda \nu_2)^2 \, d \Ha_\gamma^1  \\
&\geq \int_{\Sigma_+} \left(|x_2|  - \lambda |\nu_2|\right)^2 \, d \Ha_\gamma^1 = \int_{\Sigma_+} \left(x_2 - \lambda |\nu_2|\right)^2 \, d \Ha_\gamma^1  \\
&\geq \frac12 \int_{\Sigma_+} (x_2  - \lambda (|\nu_2|)_{\Sigma_+})^2 \, d \Ha_\gamma^1 -2 \lambda^2 \int_{\Sigma_+} (|\nu_2| - (|\nu_2|)_{\Sigma_+})^2 \, d \Ha_\gamma^1. 
\end{split}
\]  
Hence, by \eqref{extra 2 1} and $|\lambda| \leq s+1$ we deduce
\[
 \int_{\Sigma_+} (x_2  - \lambda (|\nu_2|)_{\Sigma_+})^2 \, d \Ha_\gamma^1 \leq CP_\gamma(E).
\]  
The claim then follows from
\[
\int_{\Sigma_+} (x_2  -  (x_2)_{\Sigma_+})^2 \, d \Ha_\gamma^1  \leq  \int_{\Sigma_+} (x_2  - \lambda (|\nu_2|)_{\Sigma_+})^2 \, d \Ha_\gamma^1. 
\]

\bigskip
\noindent
\textbf{Inequality \eqref{extra1}.} We choose a smooth cut-off function $\zeta : \R \to [0,1]$ such that 
\[
\zeta(t) = \begin{cases}
1 & \text{for } \, |t| \leq \frac{s}{3}\\
0 & \text{for } \,  |t| \geq \frac{s}{2}
\end{cases}
\]
and
\[
|\zeta'(t)| \leq \frac{8}{s} \qquad \text{for   } \, t \in \R.  
\]
We multiply  the equation \eqref{eq x 2}, with $\unit=e_1$,  by $x_1 \zeta^2(x_2)$ and integrate by parts 
\begin{equation} \label{extra11}
\int_{\pa E} x_1^2 \zeta^2(x_2)\, d \Ha_\gamma^1  = \int_{\pa E}\big( -\lambda x_1 \nu_1 \zeta^2(x_2) 
+\la \nabla_\tau x_1 , \nabla_\tau (x_1 \zeta^2(x_2)) \ra + \e \la b,x\ra \nu_1 x_1 \zeta^2(x_2)\big) \, d \Ha_\gamma^1.
\end{equation}
We estimate the first term on right-hand-side by Young's inequality and by $|\lambda| \leq s+1$ 
\[
-\lambda x_1 \nu_1 \zeta^2(x_2) \leq  \frac12 x_1^2 \zeta^2  + \frac{(s+1)^2}{2} \nu_1^2\zeta^2,
\]
where  we have written $\zeta = \zeta(x_2)$ for short.  We estimate  the second term by using $|\nabla_\tau \zeta(x_2)|  = |\zeta'(x_2)| |\nabla x_2| \leq \frac{8}{s}|\nu_1|$ as follows
\[
\begin{split}
\la \nabla_\tau x_1 , \nabla_\tau (x_1 \zeta^2(x_2)) \ra &\leq |\nabla_\tau x_1|^2\zeta^2 +\frac{16}{s} \zeta |x_1| |\nabla_\tau x_1||\nu_1|\\
&\leq \zeta^2 + \frac{1}{20} x_1^2\zeta^2 + \frac{C}{s^2} \nu_1^2.
\end{split}
\]
We estimate the third term  simply by using $\e |b| \leq \frac{2}{s^2}$  
\[
\e \la b,x\ra \nu_1 x_1 \zeta^2(x_2)  \leq \frac{2}{s^2}|x|^2\zeta^2.
\]
Hence, we deduce from \eqref{extra11} by the three above inequalities that  
\[
 \int_{\pa E} \left(x_1^2 -  \frac{1}{10}x_1^2 - \frac{4}{s^2} |x|^2- 2 \right) \zeta^2\, d \Ha_\gamma^1 \leq \int_{\pa E}\left(  (s+1)^2 \nu_1^2 \zeta^2 +  \frac{C}{s^2} \nu_1^2  \right) \, d \Ha_\gamma^1
\]
We recall that $\zeta = 0$ when $|x_2| \geq \frac{s}{2}$ and that  by \eqref{aux 1} 
we have that $|x|^2 \geq (s-1)^2$ on $\pa E$. In particular, for every $x \in \{ x \in \pa E : |x_2| \leq \frac{s}{2} \}$ it holds
\begin{equation} \label{extra12}
x_1^2 = |x|^2 - x_2^2 \geq (s-1)^2 - \frac{s^2}{4} \geq  \frac{3}{4}(s-2)^2
\end{equation}
and $|x|^2 \leq 2x_1^2$. Therefore we deduce
\begin{equation} \label{extra13}
\frac45  \int_{\pa E} x_1^2 \zeta^2\, d \Ha_\gamma^1 
\leq   (s+1)^2  \int_{\pa E}\nu_1^2 \zeta^2\, d \Ha_\gamma^1 +  \frac{C}{s^2} \int_{\pa E} \nu_1^2  \, d \Ha_\gamma^1.
\end{equation}

We write the first term on the right-hand-side of \eqref{extra13}  as 
\[
\begin{split}
 (s+1)^2&\int_{\pa E \cap \{ \nu_1^2 \leq 1/2 \}}   \nu_1^2 \zeta^2 \, d \Ha_\gamma^1 + (s+1)^2\int_{\pa E \cap \{ \nu_1^2 > 1/2 \}}   \nu_1^2 \zeta^2 \, d \Ha_\gamma^1\\
&\leq \frac{(s+1)^2}{2} \int_{\pa E}   \zeta^2 \, d \Ha_\gamma^1 + (s+1)^2\int_{\pa E \cap \{ \nu_1^2 > 1/2 \}}   \nu_1^2  \, d \Ha_\gamma^1\\
&\leq \frac{3}{4} \int_{\pa E}  x_1^2 \zeta^2 \, d \Ha_\gamma^1 + (s+1)^2\int_{\pa E \cap \{ \nu_1^2 > 1/2 \}}   \nu_1^2  \, d \Ha_\gamma^1,
\end{split}
\]
where the last inequality follows from \eqref{extra12}. Therefore  \eqref{extra13} implies
\[
\frac{1}{20} \int_{\pa E} x_1^2 \zeta^2\, d \Ha_\gamma^1 
\leq    (s+1)^2\int_{\pa E \cap \{ \nu_1^2 > 1/2 \}}   \nu_1^2  \, d \Ha_\gamma^1  +  \frac{C}{s^2}  \int_{\pa E} \nu_1^2  \, d \Ha_\gamma^1.
\]
Now since $|x|^2 \leq 2x_1^2$ and $\zeta(x_2) =1 $ for $|x_2| \leq \tfrac{s}{3}$ we have
\[
\int_{\pa E \cap \{ |x_2| \leq  \tfrac{s}{3} \}}|x|^2 \, d \Ha_\gamma^1 
\leq    C s^2  \int_{\pa E \cap \{ \nu_1^2 > 1/2 \}}   \nu_1^2  \, d \Ha_\gamma^1  +  \frac{C}{s^2}  \int_{\pa E} \nu_1^2  \, d \Ha_\gamma^1.
\]
Hence, we need yet to show that 
\begin{equation}   \label{extra14}
\int_{\pa E \cap \{ \nu_1^2 > 1/2 \}}   \nu_1^2  \, d \Ha_\gamma^1  \leq \frac{C}{s^2}  \int_{\pa E} \nu_1^2  \, d \Ha_\gamma^1
\end{equation}
to finish the proof of \eqref{extra1}. 

We obtain by \eqref{extra 2 ???} and  \eqref{key lemma2}  that 
\[
\begin{split}
\int_{\pa E \cap \{  \nu_1^2 > 1/2  \}} \Big||\nu_1| - \frac{\sqrt{3}}{\sqrt{7}}\Big|^2  \, d \Ha_\gamma^1 &\leq \int_{\pa E \cap \{  \nu_1^2 > 1/2  \}} \big||\nu_1| - | \bar{\nu}_1| \big |^2  \, d \Ha_\gamma^1  \\
&\leq \int_{\pa E} \big||\nu_1| - | \bar{\nu}_1| \big |^2  \, d \Ha_\gamma^1  \leq \frac{C}{s^2}  \int_{\pa E} \nu_1^2  \, d \Ha_\gamma^1.
\end{split}
\]
Thus we have
\[
\Ha_\gamma^1\left(\pa E \cap \{  \nu_1^2 > 1/2  \}\right) \leq C \int_{\pa E \cap \{  \nu_1^2 > 1/2  \}} \Big||\nu_1| 
- \frac{\sqrt{3}}{\sqrt{7}}\Big|^2  \, d \Ha_\gamma^1\leq \frac{C}{s^2}  \int_{\pa E} \nu_1^2  \, d \Ha_\gamma^1.
\]
This proves \eqref{extra14} and concludes the proof of \eqref{extra1}. 
\end{proof}

\bigskip
We are now ready to prove the reduction to the one dimensional case.

\vspace{4pt}
\begin{proof}[\textbf{Proof of Theorem \ref{2 to 1}}]
We recall that 
\[
\Sigma_+ = \{ x \in \pa E : x_2 > 0 \} \qquad \text{and} \qquad \Sigma_- = \{ x \in \pa E : x_2 < 0 \}.
\]

\smallskip
\noindent
As we mentioned in Remark \ref{why not}, using  $\varphi=\nu_e$ with $e \in \Sf^1$ orthogonal to the barycenter as a test function 
in the second variation  inequality \eqref{second var}, does not provide any information on the minimizer since  the term 
$|\int_{\partial E} \nu_e\,x\, d \Ha_\gamma^1|$ can be  too large and thus  \eqref{ineq1} becomes  trivial inequality.
We overcome this problem by essentially variating only $\Sigma_+$ while keeping $\Sigma_-$ unchanged, and vice-versa (see Figure \ref{fig1}). 
To be more precise, we restrict the class of test function by assuming $\varphi \in C^\infty(\partial E)$ to have  zero average and 
to satisfy $\varphi(x) = 0$ for every  
$x \in \pa E \cap \{ x_2 \leq - \tfrac{s}{3} \}$ (or $\varphi(x) = 0$ for  every  $x \in \pa E \cap \{ x_2 \geq \tfrac{s}{3} \}$). 
The point is that for these test function  an estimate similar to  \eqref{ineq2} holds,
\begin{equation}\label{core}
\frac{\e}{\sqrt{2\pi}}\Big| \int_{\pa E} \varphi\,x \, d \Ha_\gamma^1 \Big|^2 \leq\frac{1}{2}\int_{\pa E} \varphi^2 \, d \Ha_\gamma^1.
\end{equation}
Indeed,  by writing
\[
\begin{split}
\Big| \int_{\pa E} \varphi\,x \, d \Ha_\gamma^1 \Big|^2 &= \left(\int_{\pa E} x_1 \varphi \, d \Ha_\gamma^1 \right)^2 
+ \left(\int_{\pa E} x_2 \varphi \, d \Ha_\gamma^1 \right)^2 \\
&= \left(\int_{\pa E} (x_1- (x_1)_{\Sigma_+}) \varphi \, d \Ha_\gamma^1 \right)^2 
+ \left(\int_{\pa E} (x_2- (x_2)_{\Sigma_+}) \varphi \, d \Ha_\gamma^1 \right)^2
\end{split}
\]
and estimating both the terms by  \eqref{extra2} and \eqref{extra1} we have
\[
\begin{split}
\left(\int_{\pa E} (x_i- (x_i)_{\Sigma_+}) \varphi \, d \Ha_\gamma^1 \right)^2 
&= \left(\int_{\Sigma_+ \cup \{ |x_2|\leq \tfrac{s}{3}\}} (x_i- (x_i)_{\Sigma_+}) \varphi \, d \Ha_\gamma^1 \right)^2 \\
&\leq  8 \left(\int_{\Sigma_+ } (x_i- (x_i)_{\Sigma_+})^2 
+ \int_{\pa E \cap \{ |x_2|\leq \tfrac{s}{3}\}} |x|^2 \, d \Ha_\gamma^1 \right)\left(\int_{\pa E} \varphi^2 \, d \Ha_\gamma^1 \right)\\
&\leq C P_\gamma(E) \left(\int_{\pa E} \varphi^2 \, d \Ha_\gamma^1 \right) , \qquad \text{for } i = 1,2.
\end{split}
\]
Hence, we get \eqref{core} thanks to \eqref{bound perimeter up} and $\e \leq\frac{7\sqrt{2 \pi}}{5s^2}  e^{\frac{s^2}{2}}$  
from \eqref{choice epsilon up}.

\begin{figure} \label{pic}
\centering
\psfrag{1}{$x_2$}
\psfrag{2}{$x_1$}
\psfrag{3}{$s/3$}
\psfrag{4}{$-s/3$}
\psfrag{5}{$\Sigma_+$}
\psfrag{6}{$\Sigma_-$}
\includegraphics[width=11cm]{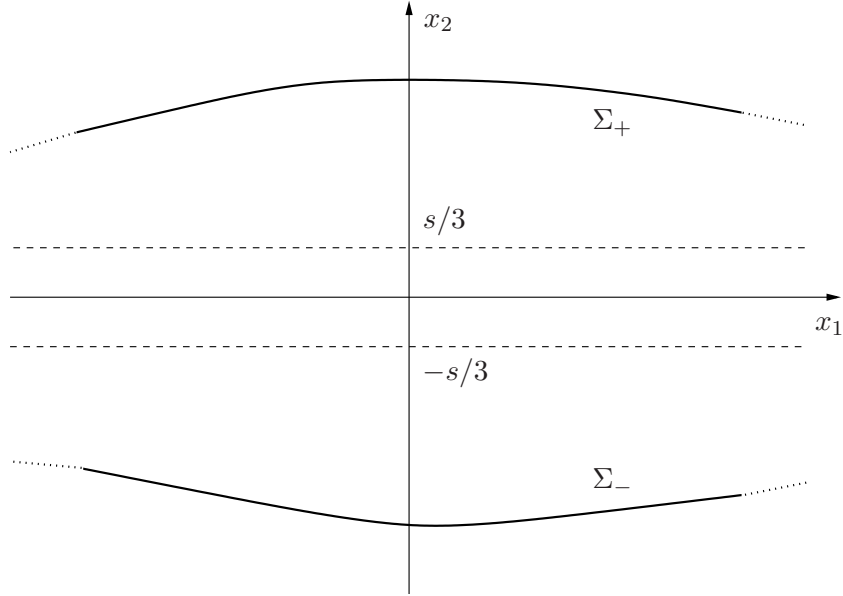}
\caption{The sets $\Sigma_+$ and $\Sigma_-$.}
\label{fig1}
\end{figure}

In order to explain  the idea of the proof, we assume first  that  $\Sigma_+$ and $\Sigma_-$ are  different components of $\partial E$. 
This is of course a major simplification but it will hopefully help the reader to follow the actual proof below.
In this case  we may use the following test functions in the second variation condition, 
\begin{equation} \label{fake testf}
\varphi_i:= \begin{cases}
\nu_i-(\nu_i)_{\Sigma_+}  \;\; &\text{on }\, \Sigma_+\\
0      \;\; &\text{on }\, \Sigma_-
\end{cases}
\end{equation}
for $i=1,2$, where $(\nu_i)_{\Sigma_+}$ is the average of $\nu_i$ on $\Sigma_+$.
We use $\varphi_i$ as a  test functions  in the second variation condition \eqref{second var} and use \eqref{core} to obtain 
\begin{equation*} 
\int_{\partial E} \left(|\nabla_\tau \varphi_i|^2 - k^2\varphi_i^2  
+ \e \langle b , \nu \rangle \varphi_i^2  - \frac12\varphi_i^2 \right)\, d\Ha^{1}_\gamma(x)  \geq 0.
\end{equation*}
By using equalities \eqref{by part1} and \eqref{by part2}, rewritten on $\Sigma_+$,  we get after straightforward calculations
\begin{equation} \label{fake 1}
(\nu_i)_{\Sigma_+}^2\int_{\Sigma_+}  (k^2 - \e \la b, \nu \ra)  \, d \Ha_\gamma^{1} 
+ \frac{1}{2}\int_{\Sigma_+} (\nu_i - (\nu_i)_{\Sigma_+})^2  \, d \Ha_\gamma^{1}  
\leq  -\e \la b, e_i \ra\int_{\Sigma_+} \nu_i \, d \Ha_\gamma^{1}, \quad i=1,2.
\end{equation}
By summing up the previous inequality for $i = 1,2$  we get 
\begin{equation*}
[(\nu_1)_{\Sigma_+}^2+(\nu_2)_{\Sigma_+}^2]\int_{\Sigma_+}  (k^2 - \e \la b, \nu \ra) \, d \Ha_\gamma^{1} 
+ \frac{1}{2}\int_{\Sigma_+} [1 - (\nu_1)_{\Sigma_+}^2 - (\nu_1)_{\Sigma_+}^2] \, d \Ha_\gamma^{1} 
\leq  -\e \int_{\Sigma_+} \la b,\nu \ra   \, d \Ha_\gamma^{1}.
\end{equation*}
This can be rewritten as 
\begin{equation*}
[1-(\nu_1)_{\Sigma_+}^2-(\nu_2)_{\Sigma_+}^2]\int_{\Sigma_+} \Bigl(\e \la b, \nu \ra+\frac{1}{2}\Bigr) \, d \Ha_\gamma^{1} 
+[(\nu_1)_{\Sigma_+}^2+(\nu_2)_{\Sigma_+}^2]\int_{\Sigma_+} k^2 \, d \Ha_\gamma^{1} 
\leq 0.
\end{equation*}
By Jensen inequality $1-(\nu_1)_{\Sigma_+}^2-(\nu_2)_{\Sigma_+}^2\geq0$, while 
$|\e \la b, \nu \ra| \leq \frac{2}{s^2}$  which follows from \eqref{usein}. Therefore $k= 0$ and $\Sigma_+$ is a line.
It is clear that a similar conclusion holds also for in $\Sigma_-$.

\medskip
When $\Sigma_+$ and $\Sigma_-$ are connected the argument is more involved, since we need a cut-off argument in order to ``separate'' 
$\Sigma_+$ and $\Sigma_-$. This is possible due to \eqref{extra1}, which implies that the perimeter of the minimizer 
in the strip $\{ |x_2| \leq s/3 \}$ is small. Therefore the cut-off argument produces an error term, which by \eqref{extra1} is 
small enough so that we may  apply the previous argument.  However, the presence of the cut-off function makes the equations more 
tangled and the estimates more complicated. Since the argument is technically involved we split the rest of the proof in two steps. 

\medskip

\noindent
\textbf{Step 1}. 
In the first step we prove 
\begin{equation} 
\label{thm 3 st 11}
\left((\nu_1)_{\Sigma_+}^2 + (\nu_2)_{\Sigma_+}^2 \right)\int_{\Sigma_+}  k^2  \, d \Ha_\gamma^{1} 
+ \int_{\Sigma_+} \left(1  - (\nu_1)_{\Sigma_+}^2 - (\nu_2)_{\Sigma_+}^2\right)   \, d \Ha_\gamma^{1}  
\leq R,
\end{equation}
where the remeinder term satisfies 
\begin{equation} 
\label{thm 3 R}
R \leq \frac{C}{s^4} \left( \frac{P_\gamma(E)^2}{\Ha_\gamma^1(\Sigma_+)} \right)\bar{\nu}_1^2.
\end{equation}

We do this by proving the counterpart of \eqref{fake 1}, which now reads as 
\begin{equation} 
\label{thm 3 st 12}
(\nu_i)_{\Sigma_+}^2\int_{\Sigma_+}  \left(\frac{k^2}{2} - \e \la b, \nu \ra\right)  \, d \Ha_\gamma^{1} 
+ \frac{1}{2}\int_{\Sigma_+} (\nu_i - (\nu_i)_{\Sigma_+})^2  \, d \Ha_\gamma^{1}  
\leq  -\e \la b, e_i \ra\int_{\Sigma_+} \nu_i \, d \Ha_\gamma^{1} + R,
\end{equation}
for $i=1,2$, where the reminder $R$ satisfies \eqref{thm 3 R}. Let us show first how \eqref{thm 3 st 11} follows from \eqref{thm 3 st 12}. 

Indeed, by $\e |b|\leq \frac{2}{s^2}$ given by \eqref{usein} we have 
\[
-(\nu_i)_{\Sigma_+}^2 \int_{\Sigma_+}  \e \la b, \nu \ra  \, d \Ha_\gamma^{1}  
\geq - \left( \fint_{\Sigma_+} \nu_i^2\, d \Ha_\gamma^{1} \right) \int_{\Sigma_+}  \e \la b, \nu \ra  \, d \Ha_\gamma^{1} 
- \frac{2}{s^2}\int_{\Sigma_+} (\nu_i- (\nu_i)_{\Sigma_+})^2  \, d \Ha_\gamma^{1}. 
\]
Therefore we have
\[
-(\nu_i)_{\Sigma_+}^2 \int_{\Sigma_+}  \e \la b, \nu \ra  \, d \Ha_\gamma^{1} 
+ \frac{1}{4}\int_{\Sigma_+} (\nu_i - (\nu_i)_{\Sigma_+})^2  \, d \Ha_\gamma^{1}    
\geq - \left( \fint_{\Sigma_+} \nu_i^2\, d \Ha_\gamma^{1} \right) \int_{\Sigma_+}  \e \la b, \nu \ra  \, d \Ha_\gamma^{1}.
\]
Thus we obtain from   \eqref{thm 3 st 12} 
\[
\begin{split}
(\nu_i)_{\Sigma_+}^2\int_{\Sigma_+}  \frac{k^2}{2}  \, d \Ha_\gamma^{1} 
-  \left( \fint_{\Sigma_+} \nu_i^2\, d \Ha_\gamma^{1} \right) \int_{\Sigma_+}  \e \la b, \nu \ra  \, d \Ha_\gamma^{1} 
+ &\frac{1}{4}\int_{\Sigma_+} (\nu_i - (\nu_i)_{\Sigma_+})^2  \, d \Ha_\gamma^{1}  \\
&\leq  -\e \la b, e_i \ra\int_{\Sigma_+} \nu_i \, d \Ha_\gamma^{1} + R.
\end{split}
\]
Note that $\sum_{i=1}^2  \int_{\Sigma_+} \la b, e_i \ra\nu_i \, d \Ha_\gamma^{1} = \int_{\Sigma_+} \la b, \nu \ra \, d \Ha_\gamma^{1}$. 
Therefore, by adding the above estimate with $i =1,2$ we obtain 
\[
\begin{split}
\left((\nu_1)_{\Sigma_+}^2 + (\nu_2)_{\Sigma_+}^2 \right)\int_{\Sigma_+}  \frac{k^2}{2}  \, d \Ha_\gamma^{1}  
- \int_{\Sigma_+}  \e \la b, \nu \ra  \, d \Ha_\gamma^1 + \frac{1}{4} &\int_{\Sigma_+} \left(\nu_1^2 
- (\nu_1)_{\Sigma_+}^2 +  \nu_2^2 - (\nu_2)_{\Sigma_+}^2 \right) \, d \Ha_\gamma^1\\
&\leq  -\int_{\Sigma_+}  \e \la b, \nu \ra  \, d \Ha_\gamma^{1} + R,
\end{split}
\]
which implies \eqref{thm 3 st 11}. Hence, we need to prove \eqref{thm 3 st 12}. 

We prove \eqref{thm 3 st 12} by using the second variation condition  \eqref{second var}  with test function 
\begin{equation*}
\varphi_i:=(\nu_i - \alpha_i)\zeta(x_2) 
\end{equation*}
for $i=1,2$.  Here $\zeta : \R \to [0,1]$ is a smooth cut-off function such that 
\begin{equation*}
\zeta(t) = \begin{cases}
1 \;\; &\text{for }\, t \geq 0,\\
0      \;\; &\text{for }\,  t \leq -s/3,
\end{cases} \qquad \text{and} \qquad |\zeta'(t)| \leq \frac{4}{s} \quad \text{for all } \, t \in \R
\end{equation*}
and $\alpha_i$ is chosen so that $\varphi_i$ has zero average. This choice is the counterpart of \eqref{fake testf} in the case 
when $\pa E$ is connected. In particular, the cut-off function $\zeta$ guarantees that $\varphi_i(x) = 0$,  for 
$x \in  \pa E \cap \{ x_2 \leq -\tfrac{s}{3} \}$. Therefore  the estimate \eqref{core} holds and the second variation condition 
\eqref{second var} yields 
\begin{equation} 
\label{thm3 st 14}
\int_{\partial E} \left(|\nabla_\tau \varphi_i|^2 - k^2\varphi_i^2  
+ \e \langle b , \nu \rangle \varphi_i^2  - \frac12\varphi_i^2 \right)\, d\Ha^{1}_\gamma(x)  \geq 0.
\end{equation}

Let us simplify the above expression. Recall that the test function is  $\varphi = (\nu_i - \alpha_i) \zeta$, where $\zeta = \zeta(x_2)$. 
By straightforward calculation
\[
\begin{split}
\int_{\partial E} |\nabla_\tau \varphi_i|^2 \, d \Ha_\gamma^1 &= \int_{\pa E} \Big( \varphi_i (- \Delta_\tau \varphi_i + \la \nabla \varphi_i , x \ra   \Big)\, d \Ha_\gamma^1 \\
&= \int_{\pa E} \Big( \varphi_i \zeta (- \Delta_\tau  \nu_i + \la \nabla_\tau \nu_i, x \ra)  + (\nu_i - \alpha_i)^2|\nabla_\tau \zeta|^2   \Big)\, d \Ha_\gamma^1.
\end{split}
\] 
Therefore we have by the above equality and by multiplying the equation  \eqref{eq nu 2} with $\varphi_i$ and integrating by parts 
\begin{equation} 
\label{thm3 st 15}
\int_{\partial E} |\nabla_\tau \varphi_i|^2 \, d \Ha_\gamma^1 =  \int_{\pa E}   (k^2 - \e \la b , \nu \ra)  \zeta^2 \nu_i (\nu_i - \alpha_i)  \, d \Ha_\gamma^1 + R_1, 
\end{equation}
where the remainder term is
\begin{equation} 
\label{thm3 st 16}
R_1 = \int_{\pa E} \left( \e \la b , e_i \ra \, \varphi_i \zeta +  (\nu_i - \alpha_i)^2|\nabla_\tau \zeta|^2 \right) \, d \Ha_\gamma^1.
\end{equation}
On the other hand, multiplying   \eqref{eq nu 2} with $\zeta^2$ and integrating by parts yields  
\begin{equation} 
\label{thm3 st 17}
\begin{split}
\alpha_i \int_{\pa E} \Big( (k^2 - \e \la b , \nu \ra)  \nu_i  \zeta^2  \Big)\, d \Ha_\gamma^1 &=\alpha_i \int_{\pa E} \Big( (-\Delta_\tau \nu_i + \la \nabla_\tau \nu_i, x \ra  ) \zeta^2   - \e \la b, e_i \ra \zeta^2 \Big)\, d \Ha_\gamma^1\\
&=    - \alpha_i  \int_{\pa E}  \e \la b, e_i \ra \zeta^2 \, d \Ha_\gamma^1 + R_2,
\end{split}
\end{equation}
where the remainder term is
\begin{equation} 
\label{thm3 st 18}
R_2 = 2\alpha_i \int_{\pa E} \zeta  \la \nabla_\tau \nu_i, \nabla_\tau \zeta  \ra \, d \Ha_\gamma^1.
\end{equation}
Collecting \eqref{thm3 st 14},  \eqref{thm3 st 15}, \eqref{thm3 st 17}  yields
\begin{equation} 
\label{thm3 st 19}
 \int_{\pa E} \left( \alpha_i^2 (k^2 - \e \la b , \nu \ra)    + \frac{1}{2}|\nu_i - \alpha_i|^2 \right)\zeta^2  \, d \Ha_\gamma^1 
 \leq -\alpha_i  \int_{\pa E} \e \la b , e_i \ra \zeta^2 \, d \Ha_\gamma^1 + R_1 + R_2,
\end{equation}
where the remainder terms $R_1$ and  $R_2$ are given by \eqref{thm3 st 16} and \eqref{thm3 st 18} respectively.

Let us next estimate the remainder terms in \eqref{thm3 st 19}. We note that   \eqref{key lemma2} (recall that $\nu_v =\nu_1$) 
implies $\fint_{\pa E} \nu_1^2 \, d \Ha_\gamma^1 \leq \left(1 + \tfrac{10}{s^2} \right) \bar{\nu}_1^2$. Therefore we deduce  
from  \eqref{aux 1} and  \eqref{extra1}   that 
\begin{equation}
\label{from extra}
\Ha_\gamma^1\big(\{ x \in \pa E : |x_2|\leq s/3\}\big) \leq \frac{C}{s^2} P_\gamma(E) \, \bar{\nu}_1^2.
\end{equation}
Therefore since $|\nabla_\tau \zeta(x)| \leq 4/s$, for $|x_2| \leq s/3$, and $\nabla_\tau \zeta(x) = 0$  otherwise, \eqref{from extra} yields
\begin{equation}
\label{zeta small}
\int_{\pa E} |\nabla \zeta|^2\,  d \Ha_\gamma^1 \leq \frac{C}{s^4} P_\gamma(E) \, \bar{\nu}_1^2.
\end{equation}
We may therefore estimate $R_2$ (given by \eqref{thm3 st 18}) by Young's inequality and  by \eqref{zeta small} as
\[
\begin{split}
R_2 &\leq   \frac{\alpha_i^2}{2} \int_{\pa E} |\nabla_\tau \nu_i|^2 \zeta^2  \, d \Ha_\gamma^1 
+  2\int_{\pa E} |\nabla_\tau  \zeta|^2  \, d \Ha_\gamma^1 \\
&\leq   \frac{\alpha_i^2}{2} \int_{\pa E} k^2 \zeta^2  \, d \Ha_\gamma^1 +  \frac{C}{s^4}P_\gamma(E) \, \bar{\nu}_1^2.
\end{split}
\]
Similarly we may estimate  \eqref{thm3 st 16} as
\[
R_1 \leq  \e \la b , e_i \ra  \int_{\pa E} \varphi_i \zeta \, d \Ha_\gamma^1  +   \frac{C}{s^4}P_\gamma(E) \, \bar{\nu}_1^2.
\]
To estimate the first term in $R_1$ we recall that $\int_{\pa E} \varphi_i  \, d \Ha_\gamma^1=0$ and therefore  
$\int_{\pa E} \varphi_i \zeta \, d \Ha_\gamma^1 = \int_{\pa E} \varphi_i (\zeta-1) \, d \Ha_\gamma^1$. 
Since  $\varphi_i (\zeta-1)  =0$ on  $\pa E \cap \{ |x_2|> s/3\}$, we deduce by $\e|b|\leq \frac{2}{s^2}$ and by \eqref{from extra} that
\[
\e \la b , e_i \ra  \int_{\pa E} \varphi_i \zeta \, d \Ha_\gamma^1  \leq   \frac{C}{s^4}P_\gamma(E) \, \bar{\nu}_1^2.
\]
Hence, we may write \eqref{thm3 st 19} as
\begin{equation} 
\label{thm3 st 1 10}
\int_{\pa E} \left( \alpha_i^2 \left(\frac{k^2}{2} - \e \la b , \nu \ra\right)    
+ \frac{1}{2}|\nu_i - \alpha_i|^2 \right)\zeta^2  \, d \Ha_\gamma^1 \leq 
-\alpha_i  \int_{\pa E} \e \la b , e_i \ra \zeta^2 \, d \Ha_\gamma^1 + \tilde{R},
\end{equation}
where the remainder term $\tilde{R}$ satisfies
\begin{equation} 
\label{thm3 st 1 11}
\tilde{R} \leq   \frac{C}{s^4}P_\gamma(E) \, \bar{\nu}_1^2.
\end{equation}

By a similar argument we may also get rid of the cut-off function in \eqref{thm3 st 1 10}. 
Indeed by  $\e|b|\leq 2/s^2$ and \eqref{from extra} we have 
$-\int_{\pa E} \e \la b , \nu \ra \zeta^2   \, d \Ha_\gamma^1 \geq -\int_{\Sigma_+} \e \la b , \nu \ra  \, d \Ha_\gamma^1 - \tilde{R}$, 
where $\tilde{R}$ satisfies \eqref{thm3 st 1 11}. 
Similarly we get $-\alpha_i  \int_{\pa E} \e \la b , e_i \ra \zeta^2 \, d \Ha_\gamma^1 \leq 
-\alpha_i  \int_{\Sigma_+} \e \la b , e_i \ra \, d \Ha_\gamma^1 +  \tilde{R}$. 
Therefore we obtain from \eqref{thm3 st 1 10}
\begin{equation} 
\label{thm3 st 1 12}
 \int_{\Sigma_+} \left( \alpha_i^2 \left(\frac{k^2}{2} - \e \la b , \nu \ra\right)    
 + \frac{1}{2}|\nu_i - \alpha_i|^2 \right) \, d \Ha_\gamma^1 \leq -\alpha_i  \int_{\Sigma_+} \e \la b , e_i \ra  \, d \Ha_\gamma^1 + \tilde{R},
\end{equation}
where the remainder term $\tilde{R}$ satisfies \eqref{thm3 st 1 11}.

We need yet to replace $\alpha_i$ by $(\nu_i)_{\Sigma_+}$ in order to obtain \eqref{thm 3 st 12}. We  do this by showing that
$\alpha_i$ is close the average $(\nu_i)_{\Sigma_+}$. To be more precise we show that 
\begin{equation}
\label{average}
|\alpha_i - (\nu_i)_{\Sigma_+}|   \leq  \frac{C}{s^2}\left(\frac{P_\gamma(E)}{\Ha_\gamma^1(\Sigma_+)}\right) \, \bar{\nu}_1^2.
\end{equation}
Indeed, since $\zeta =1$ on $\Sigma_+$ we may write
\[
\Ha_\gamma^1(\Sigma_+)(\alpha_i - (\nu_i)_{\Sigma_+})= \int_{\Sigma_+}(\alpha_i - (\nu_i)_{\Sigma_+})\zeta \,  d \Ha_\gamma^1. 
\]
Since $\zeta = 0 $ when $x_2 \leq -s/3$ we may estimate
\[
\Ha_\gamma^1(\Sigma_+) \big|\alpha_i - (\nu_i)_{\Sigma_+} \big| \leq \Big| \int_{\pa E}(\alpha_i - (\nu_i)_{\Sigma_+})\zeta \,  d \Ha_\gamma^1 \Big| 
+ 2 \Ha_\gamma^1\big(\{ x \in \pa E : |x_2|\leq s/3 \}\big).
\]
The inequality \eqref{average} then  follows from  $\int_{\pa E}(\alpha_i - (\nu_i)_{\Sigma_+})\zeta \,  d \Ha_\gamma^1 = - \int_{\pa E} \varphi_i \,  d \Ha_\gamma^1 = 0$ and from \eqref{from extra}.

We use \eqref{key lemma1} and $\e|b|\leq \frac{2}{s^2}$ to conclude that $\int_{\Sigma_+} k^2 + |\e \la b , \nu \ra|\,  d \Ha_\gamma^1  
\leq \frac{C}{s^2} P_\gamma(E)$. Therefore we may estimate \eqref{thm3 st 1 12} 
by \eqref{average} and get
\[
 \int_{\Sigma_+} \left( (\nu_i)_{\Sigma_+}^2 \left(\frac{k^2}{2} - \e \la b , \nu \ra\right)    
 + \frac{1}{2}|\nu_i - \alpha_i|^2 \right) \, d \Ha_\gamma^1 \leq -(\nu_i)_{\Sigma_+} \int_{\Sigma_+} \e \la b , e_i \ra  \, d \Ha_\gamma^1 + R,
\]
where the remainder term $R$ satisfies \eqref{thm 3 R}. Finally the inequality  \eqref{thm 3 st 12} follows from 
\[
 \int_{\Sigma_+} |\nu_i - (\nu_i)_{\Sigma_+}|^2 \, d \Ha_\gamma^1 \leq \int_{\Sigma_+} |\nu_i - \alpha_i|^2 \, d \Ha_\gamma^1.
\]

\bigskip

\noindent
\textbf{Step 2}.
Precisely similar argument as in the previous step, gives the estimate  \eqref{thm3 st 1 11} also for $\Sigma_-$, i.e., 
\begin{equation} 
\label{thm 3 st 2 1}
\left((\nu_1)_{\Sigma_-}^2 + (\nu_2)_{\Sigma_-}^2 \right)\int_{\Sigma_-}  k^2  \, d \Ha_\gamma^{1} 
+ \int_{\Sigma_-} \left(1  - (\nu_1)_{\Sigma_-}^2 - (\nu_2)_{\Sigma_+}^2\right)   \, d \Ha_\gamma^{1}  
\leq \tilde{R},
\end{equation}
where the remainder satisfies
\begin{equation} 
\label{thm 3 tildeR}
\tilde{R} \leq \frac{C}{s^4} \left( \frac{P_\gamma(E)^2}{\Ha_\gamma^1(\Sigma_-)} \right)\bar{\nu}_1^2.
\end{equation}

Let us next prove that 
\begin{equation} 
\label{thm 3 st 2 2}
\Ha_\gamma^1(\Sigma_+) \geq \frac{1}{10} P_\gamma(E) \qquad \text{and} \qquad  \Ha_\gamma^1(\Sigma_-) \geq \frac{1}{10} P_\gamma(E).
\end{equation}
Without loss of generality we may assume that $\Ha_\gamma^1(\Sigma_-) \geq \Ha_\gamma^1(\Sigma_+)$. 
In particular, we have $\Ha_\gamma^1(\Sigma_-) \geq \frac12P_\gamma(E)$ 
and therefore \eqref{thm 3 st 2 1} and \eqref{thm 3 tildeR} imply 
\begin{equation}
\label{thm 3 st 2 3}
\int_{\Sigma_-} (\nu_2 - (\nu_2)_{\Sigma_-})^2  \, d \Ha_\gamma^{1}   \leq \frac{C}{s^4}P_\gamma(E).
\end{equation}
We need to show  the first inequality in \eqref{thm 3 st 2 2}. We use \eqref{nu 2 big} and \eqref{thm 3 st 2 3} to deduce
\[
\begin{split}
\frac47P_\gamma(E) &\leq \int_{\pa E} (\nu_2 - \bar{\nu}_2)^2\, d \Ha_\gamma^{1} \leq  \int_{\pa E} (\nu_2 - (\nu_2)_{\Sigma_-})^2\, d \Ha_\gamma^{1} \\
&=\int_{\Sigma_-} (\nu_2 - (\nu_2)_{\Sigma_-})^2\, d \Ha_\gamma^{1} + \int_{\Sigma_+} (\nu_2 - (\nu_2)_{\Sigma_-})^2\, d \Ha_\gamma^{1} \\
&\leq \frac{C}{s^4}P_\gamma(E)  + 4 \Ha_\gamma^{1}(\Sigma_+).
\end{split}
\]
Hence we obtain $\Ha_\gamma^1(\Sigma_+) \geq \frac{1}{10}P_\gamma(E)$. Thus we have \eqref{thm 3 st 2 2}. 

We conclude  from \eqref{thm 3 st 2 2} and from \eqref{key lemma1} that $\fint_{\Sigma_+} k^2 \, d \Ha_\gamma^{1} \leq \frac{C}{s^2}$.
Therefore we have by \eqref{thm 3 st 11} and \eqref{thm 3 st 2 2} that 
\[
\int_{\Sigma_+} k^2 \, d \Ha_\gamma^{1} \leq \frac{C}{s^4}\, P_\gamma(E)\, \bar{\nu}_1^2.
\]
Similarly we get 
\[
\int_{\Sigma_-} k^2 \, d \Ha_\gamma^{1} \leq \frac{C}{s^4}\, P_\gamma(E)\, \bar{\nu}_1^2
\]
and therefore 
\begin{equation}
\label{thm 3 st 2 4}
\fint_{\pa E} k^2 \, d \Ha_\gamma^{1} \leq \frac{C}{s^4}\, \bar{\nu}_1^2.
\end{equation}

We are now close to finish the proof. We proceed by recalling the equation \eqref{eq nu 2}  for $\nu_1$, i.e., 
\[
\Delta \nu_1 - \la \nabla \nu_1, x\ra = -k^2\nu_1  + \e  \la b, \nu\ra \nu_1 -\e \la b, e^{(1)}\ra.
\]
We integrate this over $\pa E$, use \eqref{thm 3 st 2 4} and get
\[
- \e  \la b,  e^{(1)}\ra    + \e  \fint_{\pa E}  \la b, \nu\ra \nu_1\, d \Ha_\gamma^1   
\leq  \fint_{\pa E} k^2 \, d \Ha_\gamma^1 \leq \frac{C}{s^4} \bar{\nu}_1^2.
\]
Note that by $\bar{\nu} P_\gamma(E) =  - \sqrt{2 \pi} \, b$  (proved in \eqref{usein 2}) we have 
$|\bar{\nu}|\,  \la b,  e^{(1)}\ra = - |b| \bar{\nu}_1$.  Thus we deduce from the above inequality that
\begin{equation}
\label{thm 3 st 2 5}
\e |b| \,|\bar{\nu}_1| \leq  \e |b| \, |\bar{\nu}|\fint_{\pa E}  |\nu_1| \, d \Ha_\gamma^1+ \frac{C}{s^4} \bar{\nu}_1^2\, |\bar{\nu}|.
\end{equation}
We proceed by concluding from \eqref{nu 2 big} that 
\[
\begin{split}
\frac47 &\leq \fint_{\pa E} (\nu_2 - \bar{\nu}_2)^2\, d \Ha_\gamma^{1} \leq  \fint_{\pa E} \big((\nu_1 - \bar{\nu}_1)^2 + (\nu_2 - \bar{\nu}_2)^2\big) \, d \Ha_\gamma^{1}   = 1 - \bar{\nu}_1^2- \bar{\nu}_2^2.
\end{split}
\]
This implies 
\[
|\bar{\nu}|^2 = \bar{\nu}_1^2 + \bar{\nu}_2^2 \leq \frac37.
\]
Using this and the inequality  $\fint_{\pa E} \nu_1^2 \, d \Ha_\gamma^1 \leq \left(1 + \tfrac{10}{s^2} \right) \bar{\nu}_1^2$ 
(given by \eqref{key lemma2})  we  estimate 
\[
\begin{split}
\e |b| \, |\bar{\nu}|\, \fint_{\pa E}  |\nu_1| \, d \Ha_\gamma^1 \leq \frac{\sqrt{3}}{\sqrt{7}}\e |b| \, \left(\fint_{\pa E} \nu_1^2 \, d \Ha_\gamma^1 \right)^{1/2} \leq \frac{3}{4}\e |b| \,|\bar{\nu}_1| .
\end{split}
\]
Therefore we deduce from  \eqref{thm 3 st 2 5} 
\[
\frac14 \e |b| |\bar{\nu}_1| \leq \frac{C}{s^4} \bar{\nu}_1^2\, |\bar{\nu}|.
\]
We use   $ \e |b| \geq \frac{1}{4s^2} |\bar{\nu}|$ (from \eqref{usein})  to conclude 
\[
\frac{1}{s^2}|\bar{\nu}_1| |\bar{\nu}| \leq   \frac{C}{s^4} \bar{\nu}_1^2 \, |\bar{\nu}|. 
\]
This yields $\bar{\nu}_1 = 0$. But then \eqref{key lemma2} implies
\[
\nu_1 = 0 
\]
and we have reduced the problem to the one dimensional case.  
\end{proof}

%%%%%%%%%%%%%%%%%%%%%%%%%%%%%%%%%%%%%%%%%%%%%%%%%%%%%%%%%%%%%%%%%%%%%%%%%%%%%%%%%%%%%%%%%%%%%%%%%%%%%%%%%%%%%%%%%%%%%%%%%%%%%%%%%%%%%%%%%
\section{The one dimensional case}\label{section 1d}

\noindent
In this short section we finish the proof of the main theorem which states that the minimizer of \eqref{functional} 
is either the half-space $H_{\unit, s}$ or the symmetric strip $D_{\unit, s}$. By the previous results it is enough to 
solve the problem in the one-dimensional case. 

\begin{theorem}
\label{thm 1D} When $s$ is large enough the minimizer $E \subset \R$ of \eqref{functional} is either $(-\infty, s)$, 
$(-s,\infty)$ or $(-a(s), a(s))$.
\end{theorem}

\begin{proof}
As we explained in Section \ref{notation}, we have to prove that, when $\e$ is in the interval \eqref{choice epsilon up},
the only local  minimizers of \eqref{functional} are $(-\infty, s)$, $(-s,\infty)$ and $(-a(s), a(s))$. 
% which satisfy the perimeter bounds \eqref{bound perimeter up} and \eqref{bound perimeter low} 

Let us first show that the minimizer $E$ is an interval.  Recall that since $E \subset \R$ is a set of locally finite perimeter 
it has locally finite number of boundary points. Moreover, since there is no curvature in dimension one the Euler equation \eqref{euler} 
reads as
\begin{equation} \label{pikku euler}
- x \nu(x) + \e b x = \lambda. 
\end{equation}
By \eqref{aux 1} we have that $(-s+1, s-1) \subset E$.  
It is therefore enough to prove that the boundary $\pa E$ has at most one positive and one negative point.  
Assume by contradiction  that $\pa E$ has at least two positive points (the case of two negative points is similar). 

If $x$ is a positive point which is closest to the origin on $\pa E$ then $\nu(x) = 1$. On the other hand, if $y$ is 
the next boundary point, then $\nu(y) =-1$. Then the Euler equation yields
\[
-x + \e b x = y + \e b y.
\] 
 By $\e |b| \leq \frac{2}{s^2}$  (proved in \eqref{usein}) we conclude that 
\[
\left( 1 - \frac{2}{s^2}\right) y \leq  - \left(1  - \frac{2}{s^2}\right) x,
\] 
which is a contradiction when since $x,y >0$.

The minimizer of  \eqref{functional}  is thus an interval of the form
\[
E = (-x, y),
\]
where $s-1 \leq x, y\leq \infty$. Without loss of generality we may assume that $x \leq y$. 
Therefore we have
\begin{equation*}
e^{-\frac{x^2}{2}}\leq P_\gamma(E)\leq 2e^{-\frac{x^2}{2}}.
\end{equation*}
Using the bounds on the perimeter \eqref{bound perimeter up} and \eqref{bound perimeter low}  we conclude 
that $s-1/s \leq x \leq  s + 3/s$.  
The Euler equation \eqref{pikku euler} yields
\[
x + \e b x = y - \e b y.
\]
Hence, we conclude from $\e |b| \leq \frac{2}{s^2}$  that 
\begin{equation} \label{interval}
 s - \frac{1}{s} \leq x \leq y \leq s + \frac{8}{s}.
\end{equation}

Let us next prove that the minimizer has the volume  $\gamma(E) = \phi(s)$. Indeed, it is not possible that $\gamma(E) < \phi(s)$, 
because by enlarging $E$ we can decrease its perimeter, barycenter and the volume penalization term in \eqref{functional}. 
Also $\gamma(E) >\phi(s)$ is not possible. Indeed, in this case we can perturb the set $E$ by 
\[
E_t =  (-x +t, y), \qquad t >0.
\]  
Then $\phi(s)\leq\gamma(E_t) < \gamma(E)$ and 
\begin{equation*}\begin{split}
\frac{d}{dt} \mathcal{F}(E_t) \bigl|_{t=0} &= x e^{-\frac{x^2}{2}} + \e \, b(E)\,  x e^{-\frac{x^2}{2}} - (s +  1) e^{-\frac{x^2}{2}}\\
&\leq \left(1 +  \frac{2}{s^2}\right) x e^{-\frac{x^2}{2}}  - (s +  1) e^{-\frac{x^2}{2}},
\end{split}\end{equation*}
taking again into account that $\e |b(E)| \leq \frac{2}{s^2}$. But since $x \leq  s + 3/s$ the above  inequality yields 
$\frac{d}{dt} \mathcal{F}(E_t) \bigl|_{t=0} <0$, which contradicts the minimality of $E$. 

Let us finally show that if a local minimizer is a finite  interval  $E = (-x, y)$ for $x \leq y <\infty$, then necessarily $x = y = a(s)$. 
We study the value of the functional \eqref{functional} for intervals  
$E_t = (-\alpha(t), t)$, which have the volume $\gamma(E_t) = \phi(s)$. 
By the inequality \eqref{interval} we need to only study  the case when $a(s) \leq t \leq s + \tfrac{8}{s}$. 
This leads us to study the  function $f : [a(s), s  + \tfrac{8}{s}] \to \R$,
\[
f(t) := \mathcal{F}(E_t) = e^{-\frac{t^2}{2}} + e^{-\frac{\alpha^2(t)}{2}} 
+ \frac{\e}{2\sqrt{2\pi}} \left( e^{-\frac{\alpha^2(t)}{2}}    - e^{-\frac{t^2}{2}} \right)^2 .
\] 
The volume constraint reads as $\int_{-\alpha(t)}^t  e^{-\frac{l^2}{2}} \, dl = \sqrt{2 \pi} \, \phi(s)$. By differentiating this we obtain 
\begin{equation} \label{volume condition}
\alpha'(t) e^{-\frac{\alpha^2(t)}{2}} = -e^{-\frac{t^2}{2}} .
\end{equation}
From \eqref{volume condition} we conclude that for $t \geq \alpha(t)$ it holds $0 > \alpha'(t)  > -1$.

By differentiating $f$ once and by using \eqref{volume condition} we get
\[
f'(t) = \left( -t+\alpha(t) + \frac{\e}{\sqrt{2\pi}}( t + \alpha(t) )  
\left( e^{-\frac{\alpha^2(t)}{2}} - e^{-\frac{t^2}{2}} \right)  \right) e^{-\frac{t^2}{2}}.
\]
Therefore at a critical point it holds
\begin{equation} \label{critical 1D}
\frac{\e}{\sqrt{2\pi}}(t + \alpha(t) )  \left( e^{-\frac{\alpha^2(t)}{2}}    - e^{-\frac{t^2}{2}} \right) = t-\alpha(t).
\end{equation}
We are interested in the sign of $f''(t)$ at critical points on the interval $t \in [a(s), s  + \tfrac{8}{s}]$. 
Let us denote the barycenter of $E_t$ by 
\[
b_t := b(E_t) = \frac{1}{\sqrt{2 \pi}}  \left( e^{-\frac{\alpha^2(t)}{2}}    - e^{-\frac{t^2}{2}} \right). 
\]
By  differentiating $f$ twice and by using \eqref{volume condition} and \eqref{critical 1D} we obtain
\[
f''(t) = \left( -(1 - \e b_t) +\alpha'(t) (1 + \e b_t)  + \frac{\e}{\sqrt{2\pi}}( t + \alpha(t) )^2 e^{-\frac{t^2}{2}}   \right)e^{-\frac{t^2}{2}}
\]
at a critical point $t$. Let us write  $\e = \frac{\e_0 \sqrt{2 \pi}}{s^2} e^{\frac{s^2}{2}}$, where $\tfrac65 \leq \e_0 \leq \tfrac75$. 
In order to analyze the sign of $f''(t)$ at critical points we define $g : [a(s), s  + \tfrac{8}{s}] \to \R$ as
\[
g(t) :=  -(1 - \e b_t) +\alpha'(t) (1 + \e b_t)  + \frac{\e_0}{s^2}( t + \alpha(t) )^2 e^{-\frac{t^2}{2}}e^{\frac{s^2}{2}}.
\]
By recalling that by \eqref{interval} $\alpha(t) \geq s - 1/s$, we have 
$\e |b_t| \leq  \frac{\e}{\sqrt{2 \pi}}  e^{-\frac{\alpha^2(t)}{2}}  \leq \frac{4}{s^2}$.

Note that the end point $t =  \alpha(t) = a(s)$ is of course a critical point of $f$. Let us check that it is  a local minimum.  
We have for the barycenter $b_{a(s)}= 0$,  $\alpha'(a(s)) = -1$ by \eqref{volume condition}, $a(s) = s + \tfrac{\ln 2}{s} + o(\tfrac{1}{s})$ 
by \eqref{estimate2} and $ e^{-\frac{a(s)^2}{2}} = \frac{1}{2}\Bigl( 1 + \frac{\ln 2}{s^2} +  o(1/s^2) \Bigr) e^{-\frac{s^2}{2}}$ 
by \eqref{estimate PD}. Therefore  it holds 
\[
g(a(s)) \geq -2  + 2 \e_0 -\frac{C}{s^2} >0
\]
when $s$ is large. In particular, we deduce that $t = a(s)$ is a local minimum of $f$. 

Let us next show that $g$ is strictly decreasing. Let us  first fix a small number $\delta >0$, which value will be clear later.   
We obtain by differentiating \eqref{volume condition} that 
\[
\alpha'' =  \alpha'(\alpha\alpha' -t).
\]
By recalling that $|\alpha'(t)| \leq 1$ and that by \eqref{interval} $\alpha(t) \leq s + 8/s$,
we get that $|\alpha''(t)| \leq 2s \, |\alpha'(t)| + 16/s$ for $t \in [a(s), s  + \tfrac{8}{s}]$.
Moreover, we estimate 
$\big|\e \, b_t' \big| \leq C/s$, where $b_t'= \frac{d}{dt} b_t $. 
We may then estimate the derivative of $g$ as
\begin{equation} \label{deri g}
\begin{split}
g'(t) \leq&  \alpha''(t) (1 + \e b_t) + (1 +\alpha'(t)) \,\e \, b_t'   \\
&- \frac{\e_0}{s^2} \, t( t + \alpha(t) )^2 e^{-\frac{t^2}{2}}e^{\frac{s^2}{2}} 
+ \frac{2 \e_0}{s^2} (t + \alpha(t))(1+ \alpha'(t))e^{-\frac{t^2}{2}}e^{\frac{s^2}{2}} \\
\leq& 2s \, |\alpha'(t)|- 4 \e_0 \, s \,  e^{-\frac{t^2}{2}}e^{\frac{s^2}{2}}  + \delta
\end{split}
\end{equation} 
when  $t \in [a(s), s  + \tfrac{8}{s}]$ and $\alpha \in [s  - \tfrac{1}{s}, a(s)]$. To study \eqref{deri g}  it is convenient to  write 
\[
t = s + \frac{\ln z}{s}
\]
where $2-\delta \leq z \leq e^8$. We obtain from  the volume condition 
$\int_{-\alpha(t)}^t  e^{-\frac{l^2}{2}} \, dl = \sqrt{2 \pi} \, \phi(s)$ arguing similarly as in \eqref{estimate2} we obtain
\[
\alpha(t) =  s + \frac{1}{s} \ln \left( \frac{z}{z-1} \right)  + \frac{\varepsilon(z)}{s}
\]
and from \eqref{volume condition} that 
\[
\alpha'(t)=  - \frac{1}{z-1} +  \varepsilon(z),
\]
where $\varepsilon(z)$ is  a function which converges uniformly to zero as $s \to \infty$. Keeping these in mind we may estimate \eqref{deri g} as
\[
g'(t) \leq \frac{2s}{z-1} - \frac{4 \e_0\, s}{z} + \delta  s
\leq \frac{2s(12-7z)}{5z(z-1)} + \delta s.
\]
Since $2-\delta \leq z \leq e^8$, the above inequality shows that $g'(t)<0$ when $\delta$ is chosen small enough. 
Hence, we conclude that $g$ is strictly decreasing. 

Recall that $g(a(s))>0$. Since $g$ is strictly decreasing, there is $t_0 \in (a(s), s  + \tfrac{8}{s})$ such that 
$g(t)>0$ for $t \in [a(s),t_0)$ and $g(t)<0$ for $t \in (t_0,s  + \tfrac{8}{s}]$. Therefore the function $f$ has no other 
local minimum on $[a(s), s  + \tfrac{8}{s}]$ than the end point $t = a(s)$. Indeed, if there were another local minimum on 
$(a(s), t_0]$ there would be at least one local maximum on  $(a(s), t_0)$. This  is impossible as the previous argument shows 
that $f''(t) >0$ at every critical point on $(a(s), t_0)$. Moreover, from $g(t)<0$  for $t \in (t_0,s  + \tfrac{8}{s}]$ 
we conclude that there are no local minimum points on $(t_0,s  + \tfrac{8}{s}]$. This completes the proof. 
\end{proof}

\bigskip

%%%%%%%%%%%%%%%%%%%%%%%%%%%%%%%%%%%%%%%%%%%%%%%%%%%%%%%%%%%%%%%%%%%%%%%%%%%%%
%%%%%%%%%%%%%%%%%%RINGRAZIAMENTI

\vspace{4pt}
\noindent

\section*{Acknowledgments}
\noindent
The first author was supported by INdAM and by the project VATEXMATE.
The second author was supported by the Academy of Finland grant 314227. 

\vspace{4pt}

%%%%%%%%%%%%%%%%%%%%%%%%%%%%%%%%%%%%%%%%%%%%%%%%%%%%%%%%%%%%%%%%%%%%%%%%%%%%%%%%%%%%%%%%%%%%%%%%%%%%%%%%%%%%%%%%%%%%


\begin{thebibliography}{40}

\bibitem{BL}
D. Bakry \& M. Ledoux.
L\'{e}vy-Gromov isoperimetric inequality for an infinite dimensional
diffusion generator.
{\em Invent. Math.} 123 (1995), 259--281.

\bibitem{Bar14}
M. Barchiesi, A. Brancolini \& V. Julin.
Sharp dimension free quantitative estimates for the Gaussian isoperimetric inequality.
{\em Ann. Probab.}, 45 (2017), 668--697.

\bibitem{Bar16}
M. Barchiesi \& V. Julin.
Robustness of the Gaussian concentration inequality and the Brunn-Minkowski inequality.
{\em Calc.~Var.~Partial Differential Equations}, 56 (2017), art. n. 80.

\bibitem{Barthe01}
F. Barthe, 
An isoperimetric result for the Gaussian measure and unconditional sets. 
{\em Bull. London Math. Soc.} 33 (2001),  408--416.

 \bibitem{BW}
J. Bernstein \& L. Wang.
A sharp lower bound for the entropy of closed hypersurfaces up to dimension six. 
{\em Invent. Math.} 206 (2016), 601--627.

\bibitem{BDF}
V. B\"ogelein, F. Duzaar \& N. Fusco.
A quantitative isoperimetric inequality on the sphere. 
 {\em Adv. Calc. Var.} 10 (2017), 223--265. 

\bibitem{Bor}
C. Borell.
The Brunn-Minkowski inequality in Gauss space.
{\em Invent. Math.}  30 (1975), 207--216.

\bibitem{CR}
A. Chakrabarti \& O. Regev.
An optimal lower bound on the communication complexity of gap-Hamming-distance.  
{\em STOC'11—Proceedings of the 43rd ACM Symposium on Theory of Computing}, 51–60, ACM, New York, 2011. 

\bibitem{CMT}
R. Choksi, C.B. Muratov \& I. Topaloglu, 
An old problem resurfaces nonlocally: Gamow's liquid drops inspire today's research and applications. 
{\em Notices Amer. Math. Soc.} 64 (2017),  1275--1283. 

\bibitem{CP}
R. Choksi \& M. Peletier.
Small volume fraction limit of the diblock copolymer problem: I. Sharp-interface functional. 
{\em SIAM J. Math. Anal.} 42 (2012), 1334--1370.
 
\bibitem{CFMP}
A. Cianchi, N. Fusco, F. Maggi \& A. Pratelli.
On the isoperimetric deficit in Gauss space.
{\em Amer. J. Math.} 133 (2011), 131--186.
 
\bibitem{CIMW}
T.H. Colding, T. Ilmanen, W. P.  Minicozzi, William P\& B. White.
The round sphere minimizes entropy among closed self-shrinkers. 
{\em J. Differential Geom.} 95 (2013), 53--69. 

\bibitem{CM}
T. H. Colding \& W. P. Minicozzi.
Generic mean curvature flow I: generic singularities.
{\em Ann. of Math.} 175 (2012), 755--833. 

\bibitem{Ehr}
A. Ehrhard.
Sym\'{e}trisation dans l'espace de Gauss.
{\em Math. Scand.} 53 (1983), 281--301.

\bibitem{F}
N. Fusco. 
The quantitative isoperimetric inequality and related topics.
{\em Bull. Math. Sci.} 5 (2015), 517–607. 

\bibitem{FJ}
N. Fusco \& V. Julin.
A strong form of the quantitative isoperimetric inequality.
{\em Calc. Var. Partial Differential Equations} 50 (2014), 925--937.

\bibitem{EL}
R. Eldan.
A two-sided estimate for the Gaussian noise stability deficit. 
{\em Invent. Math.} 201 (2015), 561--624.

\bibitem{OP}
Y. Filmus, H. Hatami, S. Heilman, E. Mossel.
R. O'Donnell, S. Sachdeva, A. Wan, \& K. Wimmer.
{\em Real Analysis in Computer Science:
A collection of Open Problems}, available online (2014).
%https://simons.berkeley.edu/sites/default/files/openprobsmerged.pdf 

\bibitem{Gamov}
G. Gamow. 
Mass defect curve and nuclear constitution.
{\em  Proc. R. Soc. Lond. A}, 126 (1930),  632--644.

\bibitem{Giusti}
E. Giusti.
{\em Minimal Surfaces and Functions of Bounded Variations.}
Birkh\"auser  (1994).

\bibitem{He}
S. Heilman.
Low Correlation Noise Stability of Symmetric Sets.
Preprint (2015).

\bibitem{He2}
S. Heilman.
Symmetric convex sets with minimal Gaussian surface area
Preprint (2017).

\bibitem{vesku}
V. Julin.
Isoperimetric problem with a Coulombic repulsive term.  
{\em  Indiana Univ. Math. J.}  63 (2014),  77--89. 

\bibitem{KM}
H. Kn\"upfer \&  C. B. Muratov.
On an isoperimetric problem with a competing nonlocal term II: The general case.
{\em Comm. Pure Appl. Math.}, 67 (2014), 1974--1994.

\bibitem{LaM}
D.A. La Manna.
Local Minimality of the ball for the Gaussian perimeter. Preprint (2017). 

\bibitem{LO}
R. Latala \& K. Oleszkiewicz.
Gaussian measures of dilatations of convex symmetric sets.
{\em Ann. Probab.}, 27 (1999), 1922--1938.

\bibitem{L}
H.B. Lawson. 
{\em Lectures on minimal submanifolds. Vol. I. Second edition.}
Mathematics Lecture Series, 9. Publish or Perish, Inc., Wilmington, Del., (1980).

%\bibitem{Le}
%M. Ledoux. 
%{\em Isoperimetry and Gaussian analysis}. 
%In {\em Lectures on probability theory and statistics}, 
%volume 1648 of Lecture Notes in Math., pages 165--294.
%Springer, Berlin, (1996).

\bibitem{Ma}
F. Maggi.
{\em Sets of finite perimeter and geometric variational problems. An introduction to geometric measure theory}.
Cambridge Studies in Advanced Mathematics, 135. Cambridge University Press, Cambridge (2012).

\bibitem{SuCi}
V. N. Sudakov \& B. S. Tsirelson.
Extremal properties of half-spaces for spherically invariant measures.
{\em Zap. Nau\v{c}n. Sem. Leningrad. Otdel. Mat. Inst. Steklov. (LOMI).} 41:14--24, 165 (1974).
Problems in the theory of probability distributions, II.

\bibitem{MN}
E. Mossel \& J. Neeman.
Robust Dimension Free Isoperimetry in Gaussian Space.
{\em Ann. Probab.} 43 (2015), 971--991.
 
\bibitem{MN2}
E. Mossel \& J. Neeman.
Robust optimality of Gaussian noise stability.
{\em J. Eur. Math. Soc.} 17 (2015), 433--482.

\bibitem{Ro}
C. Rosales.
Isoperimetric and stable sets for log-concave perturbatios of Gaussian measures.
{\em Anal. Geom. Metr. Spaces} 2 (2014), 2299--3274.

\bibitem{Zhu}
J. Zhu.
On the entropy of closed hypersurfaces and singular self-shrinkers.
Preprint (2016).

\end{thebibliography}
\end{document}